\documentclass[final,3p]{elsarticle}

\usepackage{graphicx}
\usepackage{hyperref}
\usepackage[usenames,dvipsnames]{color}

\usepackage[utf8]{inputenc}
\usepackage{nomencl}
\usepackage{amsmath}
\usepackage{mathrsfs}
\usepackage{amsfonts}

\usepackage{amssymb}
\usepackage{amsthm}

\usepackage{xifthen}

\newtheorem{theorem}{Theorem}
\newtheorem{definition}[theorem]{Definition}
\newtheorem{lemma}[theorem]{Lemma}
\newtheorem{proposition}[theorem]{Proposition}

\newcommand{\skipifemptyarg}[1]{\ifthenelse{\isempty{#1}}{}{\left[#1\right]}}
\newcommand{\skipifscalar}[1]{\ifthenelse{\isempty{#1}}{}{;#1}}
\newcommand{\Lref}[1]{Lemma~\ref{#1}}
\newcommand{\Aref}[1]{Appendix~\ref{#1}}
\newcommand{\Dref}[1]{Definition~\ref{#1}}
\newcommand{\Pref}[1]{Proposition~\ref{#1}}
\newcommand{\Sref}[1]{Section~\ref{#1}}
\newcommand{\Fref}[1]{Figure~\ref{#1}}

\newcommand{\set}[1]{\mathbb{#1}}
\newcommand{\T}[1]{{\boldsymbol #1}}
\newcommand{\de}[1]{\,{\mathrm d}#1}
\newcommand{\refe}{^{(0)}}

\newcommand{\puc}{\mathcal{\lpuc}}
\newcommand{\lpuc}{Y}

\newcommand{\R}{\set{R}}
\newcommand{\C}{\set{C}}
\newcommand{\Z}{\set{Z}}
\newcommand{\X}{\set{X}}
\newcommand{\Zd}{\set{Z}^\dime}

\newcommand{\ZN}{\set{Z}_\TN}
\newcommand{\ZNd}{\set{Z}^d_\TN}

\newcommand{\Te}{\T{e}}
\newcommand{\Tefl}{\tilde{\T{e}}}
\newcommand{\ufl}{\tilde{u}}

\newcommand{\Tj}{\T{j}}
\newcommand{\TA}{\T{A}}
\newcommand{\TAref}{\T{A}\refe}
\newcommand{\TGref}{\T{\Gamma}\refe}
\newcommand{\ThGref}{\FT{\T{\Gamma}}\refe}
\newcommand{\Tu}{\T{u}}
\newcommand{\TuN}{\T{u}_\TN}
\newcommand{\TvN}{\T{v}_\TN}

\newcommand{\Tv}{\T{v}}

\newcommand{\Tf}{\T{f}}
\newcommand{\Tg}{\T{g}}

\newcommand{\TE}{\T{E}}

\newcommand{\x}{\T{x}}
\newcommand{\y}{\T{y}}
\renewcommand{\k}{\T{k}}
\newcommand{\m}{\T{m}}
\newcommand{\Tl}{\T{l}}
\newcommand{\xk}{\x^\k_\TN}
\newcommand{\xm}{\x^\m_\TN}

\newcommand{\Tnabla}{\T{\nabla}}

\newcommand{\bform}[2]{a(#1, #2)}
\newcommand{\bformN}[2]{a_\TN(#1,#2)}
\newcommand{\lform}[1]{l(#1)}
\newcommand{\lformN}[1]{l_\TN(#1)}

\newcommand{\dime}{d}
\newcommand{\E}{\mathscr{E}} 
\newcommand{\J}{\mathscr{J}}
\newcommand{\U}{\mathscr{U}}

\newcommand{\Lper}[2]{L^{#1}_\#(\puc\skipifscalar{#2})}
\newcommand{\Cper}[2]{C^{#1}_\#(\puc\skipifscalar{#2})}
\newcommand{\Wper}[3]{W^{#1,#2}_\#(\puc\skipifscalar{#3})}
\newcommand{\Hper}[2]{H^{#1}_\#(\puc\skipifscalar{#2})}

\newcommand{\FT}[1]{\widehat{#1}}
\newcommand{\meas}[1]{|#1|}

\newcommand{\bfun}[1]{\varphi_{#1}}
\newcommand{\bfunN}[1]{\bfun{\T{N},#1}}
\newcommand{\bfunDFT}[2]{\omega_\TN^{#1#2}}

\newcommand{\alp}{\alpha}
\newcommand{\bet}{\beta}

\newcommand{\imu}{\mathrm{i}}		
\newcommand{\Tz}{\T{\xi}}
\newcommand{\cA}{c_A}
\newcommand{\ch}{c_h}
\newcommand{\CA}{C_A}
\newcommand{\Ch}{C_h}
\newcommand{\rA}{\rho_A}
\newcommand{\rh}{\rho_h}

\newcommand{\OGp}[1]{\mathcal{G}\skipifemptyarg{#1}}
\newcommand{\OG}[1]{\mathcal{G}\refe\skipifemptyarg{#1}}
\newcommand{\OGadj}[1]{\mathcal{G}^{(0)\dagger}\skipifemptyarg{#1}}

\newcommand{\sprod}[3]{\bigl(#1, #2\bigr)_{#3}}
\newcommand{\norm}[2]{\bigl\| #1 \bigr\|_{#2}}
\newcommand{\normB}[2]{\Bigl\| #1 \Bigr\|_{#2}}

\newcommand{\conj}[1]{\overline{#1}}

\newcommand{\ZdmO}{\Z^\dime \backslash \{\T{0}\}}
\newcommand{\NOmO}{\set{N}^{\dime}_{0} \backslash \{\T{0}\}}

\newcommand{\z}{\xi}

\newcommand{\D}[1]{\partial_{#1}}

\newcommand{\TN}{\T{N}}
\newcommand{\TGE}{\tilde{\Te}^\mathrm{Ga}_\TN}
\newcommand{\TGEi}{\tilde{\Te}_{\TN}}

\newcommand{\PN}[1]{\mathcal{P}_\TN\skipifemptyarg{#1}}
\newcommand{\QN}[1]{\mathcal{Q}_\TN\skipifemptyarg{#1}}

\newcommand{\trn}{^\mathsf{T}}
\newcommand{\TrigN}{\mathscr{T}_\TN^{\dime}}
\newcommand{\EN}{\E_\TN}
\newcommand{\UN}{\U_\TN}
\newcommand{\JN}{\J_\TN}

\newcommand{\IN}[1]{\mathcal{I}_{\TN}\skipifemptyarg{#1}}

\newcommand{\n}{\ensuremath{{\T{n}}}}
\newcommand{\Txi}{\ensuremath{\T{\xi}}}

\newcommand{\del}{\ensuremath{\delta}}

\newcommand{\xRd}{\set{R}^{\dime}}
\newcommand{\xRdd}{\set{R}^{\dime\times \dime}}
\newcommand{\xCd}{\set{C}^{\dime}}
\newcommand{\xCdd}{\set{C}^{\dime\times \dime}}

\newcommand{\xXN}{\xRdN}

\newcommand{\xMN}{\bigl[\R^{d\times\TN}\bigr]^2}
\newcommand{\xhMN}{\bigl[\C^{d\times\TN}\bigr]^2}

\newcommand{\xRdN}{\set{R}^{\dime\times \TN}}

\newcommand{\xUN}{\set{U}_\TN}
\newcommand{\xEN}{\set{E}_\TN}
\newcommand{\xJN}{\set{J}_\TN}
\newcommand{\oplusp}{\oplus}
\newcommand{\scal}[2]{\bigl(#1,#2\bigr)}

\newcommand{\M}[1]{\mathsf{#1}} 
\newcommand{\MB}[1]{\boldsymbol{\M{#1}}}
\newcommand{\MBA}{\ensuremath{\boldsymbol{\M{A}}}}
\newcommand{\MBF}{\ensuremath{\boldsymbol{\M{F}}}}
\newcommand{\MBFi}{\ensuremath{\boldsymbol{\M{F}}^{-1}}}
\newcommand{\MBe}{\ensuremath{\boldsymbol{\M{e}}}}
\newcommand{\MBefl}{\tilde{\MBe}}
\newcommand{\MBE}{\ensuremath{\boldsymbol{\M{E}}}}
\newcommand{\MBu}{\ensuremath{\boldsymbol{\M{u}}}}
\newcommand{\MBv}{\ensuremath{\boldsymbol{\M{v}}}}

\newcommand{\MBG}{\ensuremath{{\boldsymbol{\M{G}}}}}
\newcommand{\DFT}{\MBF_\TN}
\newcommand{\iDFT}{\MBFi_\TN}
\newcommand{\MBI}{\boldsymbol{\M{I}}}
\newcommand{\MBGamm}{\MB{\Gamma}_\TN\refe}
\newcommand{\FTMBGamm}{\FT{\MB{\Gamma}}_\TN\refe}

\newcommand{\ul}[1]{{\underline{#1}}}

\newcommand{\OGD}{\MBG^{(0)}_\TN}

\newcommand{\OGDadj}{\MBG^{(0)\dagger}_\TN}
\newcommand{\OGDo}{\MBG_\TN}

\newcommand{\DQf}[1]{D^{\step}_{#1}} 
\newcommand{\DQb}[1]{D^{-\step}_{#1}} 

\newcommand{\step}{\triangle}

\DeclareMathOperator{\esssup}{ess\,sup}
\newcommand{\Tunv}[1]{\T{\epsilon}^{#1}}

\journal{Computers $\&$ Mathematics with Applications}

\begin{document}

\begin{frontmatter}



\title{An FFT-based Galerkin Method for Homogenization of Periodic Media\tnoteref{t1}}
\tnotetext[t1]{
This work was supported by the Czech Science Foundation through project  No.~P105/12/0331.  
}

\author[label1]{Jaroslav Vond\v{r}ejc\corref{cor1}}\ead{vondrejc@gmail.com}
\author[label2,label3]{Jan Zeman}\ead{
zemanj@cml.fsv.cvut.cz}
\author[label4]{Ivo Marek}\ead{ marekivo@mat.fsv.cvut.cz}

\cortext[cor1]{Corresponding author}

\address[label1]{New Technologies for the Information Society, Faculty of Applied Sciences, University of West Bohemia, Univerzitn\'{i} 2732/8, 306 14 Plze\v{n}, Czech Republic. Support was received from the project EXLIZ -- CZ.1.07/2.3.00/30.0013, which is
co-financed by the European Social Fund and the state budget of the Czech
Republic.}

\address[label2]{Department of Mechanics, Faculty of Civil Engineering, Czech Technical  University in Prague, Th\'{a}kurova 7, 166 29 Prague 6, Czech Republic.}

\address[label3]{Centre of Excellence IT4Innovations, V\v{S}B-TU Ostrava, 17.~listopadu 15/2172, 708 33 Ostrava-Poruba, Czech Republic. Support was received from the European Regional Development Fund under the IT4Innovations Centre of Excellence, project
No.~CZ.1.05/1.1.00/02.0070.}

\address[label4]{Department of Mathematics, Faculty of Civil Engineering, Czech Technical University in Prague, Th\'{a}kurova 7, 166 29 Prague 6, Czech Republic.}
 
\begin{abstract}
In 1994, Moulinec and Suquet introduced an efficient technique for the numerical resolution of the cell problem arising in homogenization of periodic media. 
The scheme is based on a fixed-point iterative solution to an integral equation of the Lippmann-Schwinger type, with action of its kernel efficiently evaluated by the Fast Fourier Transform techniques. The aim of this work is to demonstrate that the Moulinec-Suquet setting is actually equivalent to a Galerkin discretization of the cell problem, based on approximation spaces spanned by trigonometric polynomials and a suitable numerical integration scheme.
For the latter framework and scalar elliptic setting, we prove convergence of the approximate solution to the weak solution, including a-priori estimates for the rate of convergence for sufficiently regular data and the effects of numerical integration.
Moreover, we also show that the variational structure implies that the resulting non-symmetric system of linear equations can be solved by the conjugate gradient method. Apart from providing a theoretical support to Fast Fourier Transform-based methods for numerical homogenization, these findings significantly improve on the performance of the original solver and pave the way to similar developments for its many generalizations proposed in the literature.
\end{abstract}


\end{frontmatter}

\section{Introduction}\label{sec:Introduction}

In homogenization theories for periodic media, a key role is played by 
the so-called cell problem, whose structure is, in the scalar setting,
given by~\cite[Section 2.1]{Milton:2002:TTC}
\nomenclature{$\R$}{Set of real numbers}%
\nomenclature{$\dime$}{Dimension of the problem}%
\nomenclature{$\puc$}{Periodic unit cell}%

\begin{align}\label{eq:unit_cell_problem}
\Tnabla \times \Te(\x) = \T{0}, &&
\Tnabla \cdot \Tj(\x) = \T{0}, &&
\Tj(\x) = \TA(\x) \Te(\x) 
\mbox{ for }
\x \in \puc.
\end{align}
Here, $\puc \subset \xRd$ refers to the cell characterizing the microstructure
of a composite material, $\Te : \puc \rightarrow \xRd$ is the gradient field,
$\Tj : \puc \rightarrow \xRd$ denotes the flux field and the tensor
field $\TA : \puc \rightarrow \xRdd$ collects material coefficients; all
three fields must be $\puc$-periodic.
\nomenclature{$\T{a}$}{Tensor quantity}%
\nomenclature{$\T{A}$}{Second-order tensor}%
\nomenclature{$\x$}{Position vector}%
\nomenclature{$\Te$}{Gradient field}%
\nomenclature{$\Tj$}{Flux field}%
\nomenclature{$\TA$}{Coefficients}%
\nomenclature{$\Tnabla$}{Differential operator}%
The gradient field is further decomposed to
\begin{align}\label{eq:grad_field_split}
\Te(\x)
=
\TE
+
\Tefl(\x)
\mbox{ for }
\x \in \puc,
&&
\int_{\puc}
\Tefl(\x)
\de \x
=
\T{0},
\end{align}
so that $\Tefl : \puc \rightarrow \xRd$ represents the fluctuating part and $\TE
\in \set{R}^\dime$ stands for the mean applied gradient.
\nomenclature{$a\flc$}{Fluctuating part of $a$}%
\nomenclature{$\Tefl$}{Fluctuating gradient field}%
\nomenclature{$\TE$}{Macroscopic (average) gradient value}%
The usual route to solve the cell problem is to convert it
to the divergence form, by introducing an $\puc$-periodic potential $\ufl : \puc
\rightarrow \R$ satisfying $\Tefl = - \nabla \ufl$, and to obtain an
approximate solution by the finite element method, see
e.g.~\cite{Geers:2010:MSC} for a recent overview. However, this may become
computationally prohibitive, for example when the material coefficients are
defined by large data sets produced by high-resolution imaging techniques.

Exactly with such applications in mind, Moulinec and Suquet have introduced
an efficient iterative solver for problems with coefficients $\TA$ defined on a
regular grid~\cite{Moulinec:1994:FNMC}. It is based on the reformulation of the
cell problem as an integral equation of the Lippmann-Schwinger type
\begin{equation}\label{eq:Lippman-Schwinger}
\Te(\x)
+
\int_{\puc}
\TGref(\x - \y)
\bigl( \TA(\y) - \TAref\bigr)
\Te(\y)
\de\y
=
\TE
\mbox{ for }
\x \in \puc,
\end{equation}
\nomenclature{$a\refe$}{Value related to the auxiliary reference problem}%
\nomenclature{$\TGref$}{Fundamental operator}%
\nomenclature{$\TAref$}{Auxiliary data of the problem}%
where $\TGref : \puc \rightarrow \xRdd$ is the Green operator related to an
auxiliary cell problem with $\TA(\x) = \TAref \in \xRdd$, conveniently expressed
in the Fourier domain, cf.~\eqref{eq:Gamma_action}. The numerical resolution
of~\eqref{eq:Lippman-Schwinger} is based on fixed-point iterations, with the
action of $\TGref$ efficiently evaluated using the Fast Fourier Transform
techniques. The later study~\cite{Eyre:1999:FNS} revealed that the convergence
of the scheme depends on a particular choice of $\TAref$, and that the
number of iterations needed to achieve a fixed accuracy grows linearly with
the contrast in coefficients of~$\TA$~(see ahead to~\eqref{eq:A2} on
page~\pageref{eq:A2} for a precise specification).

To overcome these difficulties, several generalizations of the basic scheme have
been proposed. Eyre and Milton developed in~\cite{Eyre:1999:FNS} an accelerated
fixed-point scheme derived from a modified integral equation, possibly combined
with a multi-grid technique. The scheme was successfully extended to non-linear
problems in~\cite{vinogradov_accelerated_2008}. Another improvement due to
Michel et al.~\cite{Michel:2000:CMB,Michel:2001:CSL} is based on an equivalent
saddle-point formulation solved by the Augmented Lagrangian method, which
performs well even for the infinite-contrast case. It has also been recognized
that the original formulation~\cite{Moulinec:1994:FNMC} is equivalent to a
system of linear equations arising from a suitable discretization procedure. In
particular, Brisard and Dormieux~\cite{Brisard:2010:FFT} presented an algorithm
based on the discretization of the Hashin-Shtrikman energy functional by the
Galerkin method and studied its convergence later in~\cite{Brisard:2012:CGA}.
The approach adopted by Zeman et al.~\cite{ZeVoNoMa:2010:AFFTH} rests on the
discretization of the Lippmann-Schwinger equation~\eqref{eq:Lippman-Schwinger}
by a collocation argument; extension of this technique to the non-linear regime
has been presented by G\'{e}l\'{e}bart and Mondon-Cance
in~\cite{Gelebart:2013:NLE}. The most recent contributions include the
primal-dual formulation by Monchiet and
Bonnet~\cite{monchiet_polarization-based_2012}, and a scheme suitable for
highly-contrasted media due to Willot et al.~\cite{Willot:2013:FBS}, that is
based on a modified kernel~$\TGref$.
 
Apart from the development of more robust solvers, considerable effort has been directed towards generalizations beyond the linear setting, as well as
towards applications to real-world materials.  Such extensions were
successfully accomplished for, e.g., small-strain
elasto-plasticity~\cite{Moulinec:1998:NMC}, homogenization of shape-memory
materials~\cite{Bhattacharya:2005:MPC}, stochastic elliptic
problems~\cite{Xu:2005:SCM}, coupled multi-physics
phenomena~\cite{Brenner:2010:CAC}, or non-local damage models for quasi-brittle
materials~\cite{Li:2012:DMC}. As for the material-specific studies, these
include simulations of microstructure coarsening in tin-lead
solders~\cite{dreyer_study_2000}, modeling of
elastic~\cite{Smilauer:2006:MBM} and visco-elastic~\cite{Smilauer:2010:IVB}
response of hydrating cement pastes, full-field simulations of polycrystalline
materials~\cite{Lebensohn:2011:FFT}, multi-scale predictions for mechanical
response of multi-functional superalloys~\cite{Gaubert:2010:CPF},
ice~\cite{Montagnat:2013:MMI}, response of steels under cyclic
loading~\cite{Mareau:2013:ENS}, or high-performance cementitious
composites~\cite{daSilva:2014:MNA}, and the list is far from complete.

The present work is motivated by computational observations reported earlier by
the authors in~\cite{ZeVoNoMa:2010:AFFTH}. These are related to the
discretization of the Lippmann-Schwinger equation by the trigonometric
collocation method due to Vainikko~\cite{Vainikko:2000:FSLS}, which
consists of an expansion of the solution in terms of trigonometric
polynomials and enforcing~\eqref{eq:Lippman-Schwinger} discretely at the grid points. Such procedure results in a non-symmetric system of linear equations, which is expressed as the product of sparse structured matrices, equivalent to the original Moulinec-Suquet method~\cite{Moulinec:1994:FNMC}.
Quite surprisingly, we have observed that the system can be solved by the
standard Conjugate Gradient algorithm applicable to symmetric positive-definite
systems. Moreover, the convergence of the algorithm is independent of the
choice of $\TAref$, and the number of iterations to achieve a fixed accuracy
scales up with the square root of the contrast in coefficients $\TA$.

The goal of this paper is to demonstrate that our previous results, among
others, can be explained by recognizing that the original Moulinec-Suquet
setting is in fact equivalent to the Galerkin discretization of the weak form of
the cell problem, with approximation spaces spanned by trigonometric
polynomials. To this purpose, after presenting the adopted notation and the
necessary function spaces, in~\Sref{sec:weak} we define the weak form of the
cell problem and demonstrate its equivalence to the Lippmann-Schwinger equation
by means of a projection operator reflecting the structure
of~\eqref{eq:unit_cell_problem}. The Galerkin discretization is treated in
detail in~\Sref{sec:Discretization}, with emphasis on convergence of approximate
solutions and on the effect of numerical integration. In
\Sref{sec:Algebraic_system}, we study the properties of the system of linear
equations arising from the discretization procedure, and their relation to the
computational experiments mentioned above. \Sref{sec:conclusions} summarizes
obtained results and outlines a number of possible extensions, whilst in
\Sref{sec:Brisard} we compare outcomes of this work with the study by Brisard
and Dormieux~\cite{Brisard:2012:CGA}. Finally, in
\ref{app:trigonometric} and~\ref{app:regularity} we gather
technicalities related to approximation properties of trigonometric polynomials
and regularity of the weak solution, in order to make the paper self-contained.

\section{Notation and preliminaries}

The goal of this section is to introduce the notation and,
following~\cite{Evans:2000:PDE,Saranen:2002:PIP},
summarize the basic facts concerning the function spaces and Fourier
transform techniques utilized in the remainder of the paper.

Vectors and second-order tensors are denoted by boldface letters, e.g.
$\Tv \in \xRd$ or $\T{M} \in \xRdd$, with Greek letters used when referring to
their entries, e.g. $\T{M} = (M_{\alp\beta})_{\alp,\beta=1,\ldots,\dime}$. 
As usual, $\T{M} \Tv$, $\Tu \cdot \Tv$, and $\Tu \otimes \Tv$ refer to
\begin{align*}
\T{M} \Tv &= (M_{\alp\beta} v_\beta)_\alp, 
&
\Tu \cdot \Tv &= u_\alp v_\alp,
&
\Tu \otimes \Tv
&=
(u_\alp v_\beta)_{\alp,\beta},
\end{align*}
where we employ the summation with respect to repeated indices and assume that $\alp$ and $\bet$ standardly range from $1$ to $\dime$ for the sake of brevity.
The symbol $\delta_{\alp\beta}$ is reserved for the Kronecker delta, defined as
$\delta_{\alp\beta} = 1$ for $\alp = \beta$ and $\delta_{\alp\beta} = 0$
otherwise, so that the unit tensor is expressed as $\T{I} =
(\delta_{\alpha\beta})_{\alpha,\bet}$. To keep the notation compact, $\X$
abbreviates $\R$, $\xRd$, or $\xRdd$ and $\FT{\X}$ is used for $\C$,
$\xCd$, or $\xCdd$. We endow the spaces with the standard inner product and
norms, e.g.
\begin{align*}
\sprod{\Tu}{\Tv}{\xCd}
&= 
u_\alp \conj{v_\alp},
&
\norm{\Tv}{\xCd}^2
&= 
\sprod{\Tv}{\Tv}{\xCd},
&
\norm{\T{M}}{\xCdd}
&=
\max_{\Tv\neq\T{0}} 
\frac{\norm{\T{M}\Tv}{\xCd}}{\norm{\Tv}{\xCd}},
\end{align*}
for $\T{u},\T{v} \in\xCd$ and $\T{M}\in\xCdd$.

We consider cells in the form $\puc = \prod_\alp [-\lpuc_\alp,
\lpuc_\alp]$. Then, a function $\Tf : \xRd \rightarrow \X$ is $\puc$-periodic if
$$
\Tf(\x + \sum_\alp 2Y_\alp k_\alp \Tunv{\alp})
=
\Tf(\x)
\text{ for }
\x \in \puc
\text{ and }
\k \in \Zd,
$$
where $\Tunv{\alp} = (\delta_{\alp\beta})_\beta$ denotes the
$\alp$-th basis vector of $\xRd$. For $p \in [1,\infty]$,
\begin{align*}
\Lper{p}{\X}
=
\left\{
\Tf \in L^p_\mathrm{loc}(\xRd;\X) 
: 
\Tf \text{ is $\puc$-periodic}
\right\}
\end{align*}
denotes the space of $p$-summable $\X$-valued periodic functions. For $p \in
[1,\infty)$ these are equipped with the norm
\begin{align*}
\norm{\Tf}{\Lper{p}{\X}}^p
=
\frac{1}{\meas{\puc}}
\int_\puc
\norm{\Tf(\x)}{\X}^p\de\x,
\end{align*}
with $\meas{\puc} = 2^\dime \prod_\alp \lpuc_\alp$; for $p = \infty$ we set 
\begin{align*}
\norm{\Tf}{\Lper{\infty}{\X}}
=
\esssup_{\x\in\puc} \norm{\Tf(\x)}{\X}.
\end{align*}
For the sake of brevity, we write $\Lper{p}{}$ instead of $\Lper{p}{\R}$.
When $p=2$, $\Lper{2}{\X}$ is a Hilbert space with the scalar product
\begin{align*}
 \scal{\Tu}{\Tv}_{\Lper{2}{\X}} = \frac{1}{\meas{\puc}}\int_{\puc} \scal{\Tu(\x)}{\Tv(\x)}_{\X}\de\x.
\end{align*}

The Fourier transform of $\Tf \in \Lper{2}{\X}$ is given by
\begin{align}\label{eq:FT_def}
\FT{\Tf}( \k )
=
\conj{\FT{\Tf}(-\k)}
=
\frac{1}{\meas{\puc}}
\int_\puc
\Tf(\x)
\bfun{-\k}(\x)
\de\x
\in \FT{\X}
\text{ for }
\k \in \Zd,
\end{align}
where the functions 
\begin{align*}
\bfun{\k}(\x)
=
\exp{
\Bigl(
  \imu\pi
    \sprod{\Tz(\k)}{\x}{\xRd} 
\Bigr)}
\text{ for }
\x \in \puc 
\text{ and }
\k \in \Zd,
\end{align*}
with $\Tz(\k) = (k_\alp / \lpuc_\alp)_\alp$, form an orthonormal basis $\{\bfun{\k}\}_{\k \in \Zd}$ which span
$\Lper{2}{}$, i.e. for $\m \in \Zd$
\begin{align*}
\sprod{\bfun{\k}}{\bfun{\T{m}}}{\Lper{2}{}}
&=
\delta_{\k \m},
\end{align*}
cf.~\cite[pp.~89--91]{rudin1986real}.
Thus, every function $\Tf \in \Lper{2}{\X}$ can be expressed in the form
\begin{align*}
\Tf(\x) 
=
\sum_{\k \in \Zd}
\FT{\Tf}( \k )
\bfun{\k}(\x)
\text{ for }
\x \in \puc,
\end{align*}
and for arbitrary $\Tg \in \Lper{2}{\X}$
\begin{align}\label{eq:Plancherel}
\sprod{\Tf}{\Tg}{\Lper{2}{\X}}
=
\sum_{\k \in \Zd}
\sprod{\FT{\Tf}(\k)}{\FT{\Tg}(\k)}{\FT{\X}}.
\end{align}

The Sobolev spaces of periodic functions, $\Hper{s}{\X}$ and
$\Wper{s}{\infty}{\X}$ for $s \in [0;\infty]$ are introduced in an analogous way
to $\Lper{2}{\X}$, and $\Lper{\infty}{\X}$ and are endowed with the following
norms, cf.~\cite[p.~141]{Saranen:2002:PIP},
\begin{subequations}\label{eq:sobolev_norms}
\begin{align}\label{eq:sobolev_norm_Hs}
\norm{\Tf}{\Hper{s}{\X}}^2
&=
\normB{%
\sum_{\k\in\Zd}
\norm{\ul{\Txi}(\k)}{\xRd}^s
\FT{\Tf}(\k)
\bfun{\k}
}{\Lper{2}{\X}}^2
= 
\sum_{\k\in\Zd} 
\norm{\ul{\Txi}(\k)}{\xRd}^{2s} 
\norm{\FT{\Tf}(\k)}{\FT{\X}}^2,
\\
\label{eq:sobolev_norm_Wsinfty}
\norm{\Tf}{\Wper{s}{\infty}{\X}}
& =
\normB{%
\sum_{\k\in\Zd}
\norm{\ul{\Txi}(\k)}{\xRd}^s
\FT{\Tf}(\k)
\bfun{\k}
}{\Lper{\infty}{\X}},
\end{align}
\end{subequations}
where 
\begin{align*}
\ul{\Txi}(\k) 
=
\begin{cases}
\Tz(\k) & \text{for } \k\in\ZdmO,
\\
\T{1} & \text{for }\k=\T{0}.
\end{cases}
\end{align*}
The space of $s$-times differentiable periodic functions $\Cper{s}{\X}$ is
understood in a similar way, so that $\Hper{s}{X} \subset \Cper{0}{X}$ for $s >
\dime / 2$ by the Sobolev inequalities,
e.g.~\cite[Section~5.6.3]{Evans:2000:PDE}. We also abbreviate $\Hper{s}{\R}$ to
$\Hper{s}{}$, or $\Cper{s}{\R}$ to $\Cper{s}{}$. Additional notation is
introduced when needed.

\section{Weak and integral formulations of cell problem}
\label{sec:weak}

Having introduced the general notation, we now proceed with the formulation
of the Lippmann-Schwinger equation~\eqref{eq:Lippman-Schwinger} and of the weak
form of the cell
problem~\eqref{eq:unit_cell_problem}--\eqref{eq:grad_field_split} in a
rigorous way. By introducing a suitable projection operator in
\Sref{sec:projection_operator}, we then show that these two formulations are
equivalent, thereby establishing a convenient discretization framework for the
following sections.

\subsection{Problem setting}\label{sec:problem_setting}

Given the structure of the cell problem~\eqref{eq:unit_cell_problem}, we begin
with the definition of the divergence and curl operators for $\Tu
\in \Lper{2}{\xRd}$, understood as periodic distributions\footnote{Here,
$\Hper{-1}{\xRdd}$ and $\Hper{-1}{}$ standardly denote the spaces of all linear
functionals on $\Hper{1}{\xRdd}$ and $\Hper{1}{}$.} $\nabla\times\Tu\in
\Hper{-1}{\xRdd}$ and $\nabla\cdot\Tu\in \Hper{-1}{}$ satisfying
\begin{subequations}
\begin{align}
\sprod{(\Tnabla \times \Tu)_{\alp\beta}}{v}{\Lper{2}{}}
& =
- 
\frac{1}{\meas{\puc}}
\int_\puc
\bigl( 
u_\alpha(\x) \D{\bet} v(\x)
-
u_\beta(\x) \D{\alp} v(\x) 
\bigr)
\de\x,
\label{eq:curl_def}
\\
\sprod{\Tnabla \cdot \Tu}{v}{\Lper{2}{}}
& =
- 
\frac{1}{\meas{\puc}}
\int_\puc
\Tu(\x) \cdot \Tnabla v(\x)
\de \x,
\end{align}
\end{subequations}
for all $v \in \Hper{1}{}$, $\D{\alp}$ denoting the weak derivative,
cf.~\cite[pp.~2--3]{Jikov:1994:HDO}. It will also be useful to consider the
spaces of zero-mean curl- and divergence-free fields
\begin{subequations}
\begin{align}
\E & 
=
\Bigl\{ 
\Tu \in \Lper{2}{\xRd} 
: 
\Tnabla \times \Tu = \T{0}
,
\int_\puc \Tu(\x) \de\x
=
\T{0}
\Bigr\},
\label{eq:E_def}
\\
\J &
=
\Bigl\{
\Tu \in \Lper{2}{\xRd}
:
\Tnabla \cdot \Tu = 0,
\int_\puc \Tu(\x) \de \x = \T{0}
\Bigr\}. 
\end{align}
\end{subequations}
\nomenclature{$\E$}{Set of curl-free periodic vectors}%
\nomenclature{$\J$}{Set of divergence-free periodic vectors}%
By virtue of the Helmholtz decomposition of periodic functions,
$\Lper{2}{\xRd}$ admits an orthogonal decomposition in the form,
e.g.~\cite[pp.~6--7]{Jikov:1994:HDO}
or~\cite[Section 12.1]{Milton:2002:TTC}
\begin{align}\label{eq:Helmholtz}
\Lper{2}{\xRd}
=
\U \oplusp \E \oplusp \J,
\end{align}
where $\U$ collects the constant fields, and 
$\oplusp$ denotes the direct sum of mutually
orthogonal subspaces.

As for the coefficients, we assume that they are essentially
bounded,
\begin{align}\tag{A1}\label{eq:A1}
\TA \in \Lper{\infty}{\xRdd},
\end{align}
symmetric and uniformly elliptic, so that there
exist constants $0 < \cA \leq \CA < +\infty$ such that 
\begin{align}\tag{A2}\label{eq:A2}
A_{\alp\bet}(\x) = A_{\bet\alp}(\x),
&&
\cA \norm{\Tv}{\xRd}^2
\leq 
\scal{\TA(\x)\Tv}{\Tv}_{\xRd}
\leq \CA \norm{\Tv}{\xRd}^2,
\end{align}
a.e.~in $\puc$ for all $\Tv \in \xRd$. By $\rA = \CA / \cA$ we denote the
condition number of $\TA$ quantifying the contrast in coefficients.
Similarly, for the auxiliary tensor $\TAref \in\xRdd$
in~\eqref{eq:Lippman-Schwinger} we assume that
\begin{align}\tag{A3}\label{eq:A3}
A\refe_{\alp\bet} = A\refe_{\bet\alp},
&&
\cA\refe 
\norm{\Tv}{\xRd}^2
\leq \scal{\TA\refe \Tv}{\Tv}_{\xRd} \leq \CA\refe \norm{\Tv}{\xRd}^2
\end{align}
for all $\Tv \in \xRd$ with $0 < \cA\refe \leq \CA\refe < +\infty$ and set
$\rA\refe = \CA\refe / \cA\refe$. 

The bilinear $a : \Lper{2}{\xRd} \times \Lper{2}{\xRd} \rightarrow \R$ and
linear $b: \Lper{2}{\xRd} \rightarrow \R$ forms associated with the cell problem
are defined as
\nomenclature{$\bform{\bullet}{\bullet}$}{Bilinear form}%
\nomenclature{$\lform{\bullet}$}{Linear form}%
\begin{align}\label{eq:forms_def1}
\bform{\Tu}{\Tv}
=
\sprod{\TA\Tu}{\Tv}{\Lper{2}{\xRd}},
&&
\lform{\Tv}
=
-
\sprod{\TA\TE}{\Tv}{\Lper{2}{\xRd}}.
\end{align}
Under the assumptions \eqref{eq:A1} and~\eqref{eq:A2}, they meet the
standard conditions of coercivity and boundedness, i.e.
\begin{subequations}\label{eq:form_estimates}
\begin{align}
\bform{\Tu}{\Tu} 
\geq
\cA
\norm{\Tu}{\Lper{2}{\xRd}}^2,
&&
| \bform{\Tu}{\Tv} |
& \leq 
\CA
\norm{\Tu}{\Lper{2}{\xRd}} \norm{\Tv}{\Lper{2}{\xRd}},
\\
&&
| \lform{\Tv} | 
& \leq
\CA 
\norm{\TE}{\R^{\dime}}
\norm{\Tv}{\Lper{2}{\xRd}}
\end{align}
\end{subequations}
for all $\Tu, \Tv \in \E$. Finally, the action of the operator $\TGref$ follows
from
\begin{align}\label{eq:Gamma_action}
\TGref[\Tu](\x)
=
\int_{\puc}
\TGref(\x - \y)
\Tu(\y)
\de \y
=
\sum_{\k \in \Zd}
\ThGref(\k)
\FT{\Tu}(\k)
\bfun{\k}(\x),
\end{align}
for $\Tu \in \Lper{2}{\xRd}$ and 
\begin{align} \label{eq:Gamma_hat}
\ThGref(\k) &= 
 \begin{cases}
\displaystyle
\frac{%
\Tz(\k) \otimes \Tz(\k)
}{%
\scal{\TA\refe \Tz(\k) }{\Tz(\k)}_{\xRd}}
   &\text{for }\k\in\ZdmO,
   \\
   \T{0} \otimes \T{0} &\text{for }\k=\T{0}.
 \end{cases}
\end{align}
Now, the solutions introduced earlier in
\Sref{sec:Introduction} are provided by the following
\nomenclature{$\ZdmO$}{Punctuated lattice}%
\begin{definition}\label{def:solutions} A field $\Tefl \in \E$ is a weak
solution to the cell problem if
\begin{align}\label{eq:weak_solution}\tag{Ws}
\bform{\Tefl}{\Tv} 
=
\lform{\Tv}
\text{ for all }
\Tv 
\in
\E.
\end{align}
Moreover, a solution to the Lippmann-Schwinger equation $\Te \in \Lper{2}{\xRd}$
satisfies
\begin{align}\label{eq:integral_solution}\tag{L-S}
\Te(\x)
+
\int_{\puc}
\TGref(\x - \y)
\bigl( \TA(\y) - \TAref\bigr)
\Te(\y) 
\de\y
=
\T{E}
\text{ a.e. in } \puc.
\end{align}
\end{definition}
\nomenclature{$\Tv$}{Testing function}%

\subsection{Projection operator}\label{sec:projection_operator}
Clearly, the concept of the weak solution~\eqref{eq:weak_solution} is more
convenient for further analysis, but this come at the expense of involving
a rather complex space $\E$ defined in~\eqref{eq:E_def}. To overcome this difficulty, we
introduce an auxiliary operator
\begin{align}
\OG{\Tu}(\x)
=
\TGref[\TAref\Tu](\x)
=
\int_{\puc}
\TGref(\x - \y)
\TAref
\Tu(\y)
\de \y
\end{align}
and set $\OGp{} = \OG{}$ for $\TA\refe=\T{I}$. The next lemma summarizes
their main properties.

\begin{lemma}\label{lem:projection}
Let~\eqref{eq:A3} be satisfied. Then, the following statements hold:
\begin{itemize}
  \item[(i)] $\OG{}$ is a well-defined bounded operator 
  $\Lper{2}{\xRd} \rightarrow \Lper{2}{\xRd}$,
  \item[(ii)] the adjoint operator to $\OG{}$ is given by $\OGadj{} =
  \TAref \TGref{}$,
  \item[(iii)] $\OG{}$ is a projection onto $\E$,
  \item[(iv)]  $\OG{\Tu} = \T{0}$ for all $\Tu \in \U$ and 
  $\sprod{\OG{\Tu}}{\Tv}{\Lper{2}{\xRd}}
  = 0$ for all $\Tu \in \Lper{2}{\xRd}$ and $\Tv \in \J \oplusp
  \U$,
  \item[(v)] for $\TAref = \lambda\T{I}$
  with $\lambda > 0$, $\OG{}$ becomes an orthogonal projection $\OGp{}$ independent of $\TAref$.
\end{itemize}
\end{lemma}

\begin{proof}
To simplify the notation, set
$$
\|\Tz(\k)\|_{\TA\refe}^2
=
\scal{\TA\refe \Tz(\k) }{\Tz(\k)}_{\xRd}
\text{ for } 
\k \in \ZdmO,
$$
with the properties $\|\Tz(\k)\|_{\TA\refe}^2 = \|\Tz(-\k)\|_{\TA\refe}^2$ and
$\|\Tz(\k)\|_{\TA\refe}^2 \geq \cA\refe \|\Tz(\k)\|_{\xRd}^2 > 0$. To prove (i),
first observe that, by~\eqref{eq:FT_def}, $\OG{}$ maps a real-valued input
to a real-valued output,
\begin{align*}
\conj{\FT{\OG{\Tu}}(\k)}
=
\frac{\Tz(\k) \otimes \Tz(\k)}{\|\Tz(\k)\|_{\TA\refe}^2}
\TAref
\conj{\FT{\Tu}}(\k)
=
\frac{\Tz(-\k) \otimes \Tz(-\k)}{\|\Tz(-\k)\|_{\TA\refe}^2}
\TAref
\FT{\Tu}(-\k)
=
\FT{\OG{\Tu}}(-\k),
\end{align*}
for all $\k \in \ZdmO$. A standard estimate
$$
\norm{\OG{\Tu}}{\Lper{2}{\xRd}}
\leq
\rA\refe
\norm{\Tu}{\Lper{2}{\xRd}}
$$
then implies the boundedness of $\OG{}$.

Proceeding to~(ii), we recall that $\OGadj{}$ is defined by 
\begin{align*}
\sprod{\Tv}{\OG{\Tu}}{\Lper{2}{\xRd}}
=
\sprod{\OGadj{\Tv}}{\Tu}{\Lper{2}{\xRd}}
\text{ for all }
\Tu, \Tv \in \Lper{2}{\xRd}.
\end{align*}
In view of the Plancherel theorem~\eqref{eq:Plancherel} and
\eqref{eq:Gamma_action}, (ii)~requires the relation
\begin{align*}
\sum_{\alp,\bet,\gamma}
\FT{v}_\alp(\k)
\frac{%
\z_\alpha(-\k) \z_\bet(-\k)
}{%
\|\Tz(-\k)\|_{\TA\refe}^2
}
A\refe_{\bet\gamma}
\FT{u}_\gamma(-\k)
=
\sum_{\alp,\bet,\gamma}
A\refe_{\gamma\alpha}
\frac{%
\z_\alp(\k) \z_\bet(\k)
}{%
\|\Tz(\k)\|_{\TA\refe}^2
}
\FT{v}_\beta(\k)
\FT{u}_\gamma(-\k)
\end{align*}
to hold for all $\k \in \ZdmO$. This is a direct consequence of the symmetry of
$\TAref$ required in~\eqref{eq:A3}.

As for (iii), we first prove that $\OG{}$ is a projection,
i.e. $\OG{\OG{\Tu}} = \Tu$. By~\eqref{eq:Gamma_action},
\begin{align*}
\OG{\OG{\Tu}}(\x)
& =
\OG{}\Bigl[ 
\sum_{\k \in \ZdmO}
\frac{%
\Tz(\k) \otimes \Tz(\k)
}{%
\|\Tz(\k)\|_{\TA\refe}^2
}
\TAref
\FT{\Tu}(\k)
\bfun{\k}(\x)
\Bigr]
\\
& =
\sum_{\k \in \ZdmO}
\frac{%
\Tz(\k) \otimes \Tz(\k)
}{%
\|\Tz(\k)\|_{\TA\refe}^2
}
\TAref
\frac{%
\Tz(\k) \otimes \Tz(\k)
}{%
\|\Tz(\k)\|_{\TA\refe}^2
}
\TAref
\FT{\Tu}(\k)
\bfun{\k}(\x).
\end{align*}
The projection properties of $\OG{}$ are the direct consequence of the identity
\begin{align*}
\bigl[ \Tz(\k) \otimes \Tz(\k) \bigr]
\TAref
\bigr[ \Tz(\k) \otimes \Tz(\k) \bigr]
&=
\|\Tz(\k)\|_{\TA\refe}^2
\Tz(\k) \otimes \Tz(\k),
\end{align*}
valid for all $\k \in \ZdmO$. To verify that $\Tnabla \times \OG{\Tu} = \T{0}$,
we transfer~\eqref{eq:curl_def} to the Fourier space
via~\eqref{eq:Plancherel} to obtain, for all $v \in \Hper{1}{}$,
\begin{align*}
\sprod{(\Tnabla \times \OG{\Tu}&)_{\alpha\beta}}{v}{\Lper{2}{}}
= 
\\
&= -
\sum_{\k\in\Zd}
\FT{(\OG{\Tu})}_\alpha(\k) 
\FT{\D{\bet} v}(-\k)
-
\FT{(\OG{\Tu})}_\beta(\k) 
\FT{\D{\alp} v}(-\k)
\\
&= 
\sum_{\k\in\ZdmO}
\frac{\imu \pi}{\|\Tz(\k)\|_{\TA\refe}^2}
\sum_{\gamma,\delta}
{\z_\alp(\k) \z_\gamma(\k)}
A\refe_{\gamma\delta}
\FT{u}_{\delta}(\k)
\z_\bet(-\k)
\FT{v}(-\k)
\\
 &\quad-
\sum_{\k\in\ZdmO}
\frac{\imu \pi}{\|\Tz(\k)\|_{\TA\refe}^2}
\sum_{\gamma,\delta}
\z_\bet(\k) \z_\gamma(\k)
A\refe_{\gamma\delta}
\FT{u}_{\delta}(\k)
\z_\alp(-\k)
\FT{v}(-\k)
= 0,
\end{align*}
where we have utilized that
$
\FT{\D{\alp} v}(\k)
=
\imu \pi 
\z_\alp(\k)
\FT{v}(\k)
$.
Recognizing that the zero-mean property in~\eqref{eq:E_def} is satisfied by
excluding $\k = \T{0}$ from the sum~\eqref{eq:Gamma_action}, this proves that
$\OG{}$ is a projection into $\E$. The surjectivity follows from its
$\E$-invariance, i.e
$
\OG{\Tu} = \Tu 
\text{ for all }
\Tu \in \E.
$
Indeed, since for any $\Tu \in \E$ there exist $f \in \Hper{1}{}$
such that $\Tu = \Tnabla f$, e.g.~\cite[p.~6]{Jikov:1994:HDO} or \cite[p.~98]{Vondrejc:2013:FFT} for a proof, we proceed
analogously to the previous step to get, for all $v \in \Lper{2}{}$,
\begin{align*}
\sprod{
\D{\alp} f
-
( \OG{\Tnabla f}& )_\alpha
}{%
v
}{\Lper{2}{}} =
\\
& = 
\sum_{\k \in \ZdmO}
\Bigl( 
\FT{\D{\alp} f}(\k)
-
\sum_{\beta,\gamma}
\frac{\z_\alpha(\k) \z_\beta(\k)}%
{\|\Tz(\k)\|_{\TA\refe}^2}
A\refe_{\beta\gamma}
\FT{\D{\gamma} f}(\k)
\Bigr)
\FT{v}(-\k)
\\
& = 
\imu \pi
\sum_{\k \in \ZdmO}
\bigl( \z_\alp(\k) - \z_\alp(\k) \bigr)
\FT{f}(\k)
\FT{v}(-\k)
=
0.
\end{align*}

Item~(iv) directly follows from definition of $\OG{}$ and from the
Helmholtz decomposition~\eqref{eq:Helmholtz}, respectively. Finally, (iv)~a
consequence of the fact that $\lambda$ cancels out in the definition of $\OG{}$.
Hence, $\OGp{}$ becomes self-adjoint and thus orthogonal, e.g.~\cite[Theorem
12.14]{Rudin:1973:FA}.
\end{proof}

\subsection{Equivalence of solutions}

With the results of \Lref{lem:projection} in hand, we are now in the position to
state the main result of this section.
\begin{proposition}\label{prop:equivalence}
Let \eqref{eq:A1}--\eqref{eq:A3} hold. Then, the solutions $\Tefl$ and $\Te$
from \Dref{def:solutions} exist and are unique. Moreover, they are equivalent,
in the sense that
\begin{align}\label{eq:Ws_LS_equiv}
\Te = \TE + \Tefl.
\end{align}
\end{proposition}

\begin{proof}
Existence and uniqueness of the weak solution is ensured by the Lax-Milgram
theorem, e.g.~\cite[Section~6.2.1]{Evans:2000:PDE}, and
estimates~\eqref{eq:form_estimates}.
Now we demonstrate that the solution to~\eqref{eq:weak_solution} is also
a solution to~\eqref{eq:integral_solution}, thereby proving existence of the
latter. We start from the explicit expression of the weak form as
\begin{align*}
\sprod{\TA\Tefl}{\Tv}{\Lper{2}{\xRd}}
=
-
\sprod{\TA\TE}{\Tv}{\Lper{2}{\xRd}}
\text{ for all } \Tv \in \E.
\end{align*}
By \Lref{lem:projection}(iii), this entails that 
\begin{align*}
\sprod{\TA\Tefl}{\OG{\Tv}}{\Lper{2}{\xRd}}
=
-
\sprod{\TA\TE}{\OG{\Tv}}{\Lper{2}{\xRd}}
\text{ for all } \Tv \in \Lper{2}{\xRd}.
\end{align*}
Utilizing \Lref{lem:projection}(ii), we further deduce 
\begin{align*}
\sprod{\TAref\TGref[\TA\Tefl]}{\Tv}{\Lper{2}{\xRd}}
=
-
\sprod{\TAref \TGref[\TA\TE]}{\Tv}{\Lper{2}{\xRd}}
\text{ for all } \Tv \in \Lper{2}{\xRd},
\end{align*}
so that 
\begin{align*}
\TAref \TGref\bigl[\TA(\TE+\Tefl)\bigr]
=
\T{0}
\text{ a.e. in }
\puc. 
\end{align*}
Multiplying the previous relation from left by $(\TAref)^{-1}$, we find it to be
equivalent to~\eqref{eq:integral_solution} provided that the following identity
holds:
\begin{align*}
( \TE + \Tefl )
-
\OG{\TE + \Tefl}
=
\TE
\text{ a.e. in }
\puc.
\end{align*}
But this is an easy consequence of \Lref{lem:projection}(iv). 

Now we demonstrate the uniqueness of solution \eqref{eq:integral_solution} by
showing that all such solutions satisfy the weak formulation
\eqref{eq:weak_solution}. Indeed, take an $\Te \in \Lper{2}{\xRd}$ satisfying
\eqref{eq:integral_solution} and decompose it into orthogonal components $\E$
and  $\U\oplusp\J$, i.e. $\Te = \OGp{}\Te + (\T{I}-\OGp{})\Te$. Then proceeding
in the reverse order as in the previous part of the proof, we obtain the
solution equivalence \eqref{eq:Ws_LS_equiv} with the unique $\OGp{}\Te = \Tefl$
and $(\T{I}-\OGp{})\Te = \TE$.
\end{proof}

\section{Discretization}
\label{sec:Discretization}

As already noted, the variational form of~\eqref{eq:weak_solution} makes it
well-suited to the discretization by the Galerkin projection onto a suitable
finite-dimensional subspace. In our setting, it turns out that a convenient
choice is the space of trigonometric polynomials, properties of which are
summarized in \Sref{sec:trig_polynomials} following the exposition of Saranen
and Vainikko~\cite[Chapter 8]{Saranen:2002:PIP}.
Convergence of such approximate solutions is studied in 
\Sref{sec:conforming_Galerkin}, utilizing the well-known techniques developed
for the analysis of the finite element method. Finally, in
\Sref{sec:non_conforming_Galerkin}, we extend these results to cover the effects
of numerical integration.

\subsection{Trigonometric polynomials}\label{sec:trig_polynomials}
Consider the cell $\puc$ discretized with a regular grid of $N_1 \times
\ldots \times N_\dime$ points, located at 
$$
\xk
=
\begin{bmatrix} 
k_1 h_1 & \ldots & k_\dime h_\dime 
\end{bmatrix}\trn
\text{ for } 
\k \in \ZNd
=
\left\{
\m \in \Zd 
:
- \frac{N_\alp}{2} \leq m_\alp < \frac{N_\alp}{2}
\right\},
$$
where $h_\alp = 2 \lpuc_\alp / N_\alp$ correspond to grid spacings in individual directions. For brevity, we shall denote $\TN = (N_1, \ldots, N_\dime)$, $|\TN |
= \prod_\alp N_\alp$ and, similarly to \Sref{sec:problem_setting}, set $\ch =
\min_{\alp} h_\alp$, $\Ch = \max_{\alp} h_\alp$, and $\rh =
\Ch / \ch$.

To keep our exposition transparent, 
we require the grid to be symmetric with respect to the origin, i.e.
\begin{align}\label{eq:A4}
N_\alp \text{ is odd for }
\alp = 1, \ldots, \dime,
\tag{A4}
\end{align}
so that symmetry of the Fourier transform of the real-valued functions,
recall~\eqref{eq:FT_def}, can be easily preserved in the discrete setting. The
generic case is elaborated in detail in~\cite[p.~126--130]{Vondrejc:2013:FFT} and will be
reported separately.

Now, the space of $\xRd$-valued trigonometric polynomials can be defined as
\begin{align}\label{eq:trig_fourier}
\TrigN
=
\Bigl\{
\sum_{\k\in\ZNd}
\FT{\Tv}^{\k}
\bfun{\k}
:
\FT{\Tv}^{\k}\in \xCd,
\FT{\Tv}^{\k} 
=
\overline{(\FT{\Tv}^{-\k})} 
\Bigr\}
\subset 
\Cper{\infty}{\xRd}
\end{align}
Every trigonometric polynomial $\Tu_\TN \in \TrigN$ admits an expression in
terms of its grid values
\begin{align}\label{eq:trig_real}
\Tu_\TN( \x ) 
=
\sum_{\k \in \ZNd}
\Tu_\TN( \xk )
\bfunN{\k}(\x)
\text{ for }
\x \in \puc,
\end{align}
where 
\begin{equation}
\bfunN{\k}(\x)
=
 \frac{1}{|\TN|}
  \sum_{\m \in \ZNd}
  \exp 
  \left\{  
    \imu \pi \sum_{\alp} 
    m_\alpha \left( 
      \frac{x_\alp}{\lpuc_\alp} - \frac{2k_\alpha}{N_\alpha} 
    \right) 
  \right\}
\text{ for }
\k \in \ZNd,
\end{equation}
are the fundamental trigonometric polynomials satisfying 
\begin{align}\label{eq:fund_trig_polynomial_property}
\bfunN{\k}( \xm ) = \delta_{\k\m}, &&
\sprod{\bfunN{\k}}{\bfunN{\m}}{\Lper{2}{}}
=
\frac{1}{\meas{\TN}}
\delta_{\k\m},
\end{align}
for all $\k, \m \in \ZNd$. Therefore, for arbitrary $\Tu_\TN, \Tv_\TN \in \TrigN$,
\begin{align}\label{eq:scal_trig_pol}
\sprod{\Tu_\TN}{\Tv_\TN}{\Lper{2}{\xRd}}
=
\frac{1}{|\TN|}
\sum_{\k \in \ZNd}
\sprod{\Tu_\TN( \xk )}{\Tv_\TN( \xk )}{\xRd}.
\end{align}

We recall that both representations are connected by means of the
Discrete Fourier Transform:
\begin{align}
\FT{\Tu_\TN}(\k)
=
\frac{1}{\meas{\TN}}
\sum_{\m \in \ZNd}
\Tu_\TN(\xm)
\bfunDFT{-\k}{\m},
&&
\Tu_\TN(\xk)
=
\sum_{\m \in \ZNd}
\FT{\Tu_\TN}( \m )
\bfunDFT{\k}{\m},
\end{align}
for $\k \in \ZNd$ and 
\begin{align}\label{eq:DFT_def}
\bfunDFT{\k}{\m}
=
\exp   
\left(2 \pi \imu\sum_{\alp} 
     \frac{k_\alp m_\alp}{N_\alp}  
\right).
\end{align}

Two projection operators, based on relations~\eqref{eq:trig_fourier}
and~\eqref{eq:trig_real}, will be used extensively in what follows,
cf.~\cite[Chapter~8]{Saranen:2002:PIP}. First, the truncation
operator $\PN{} : \Lper{2}{\xRd} \rightarrow \TrigN$ defined as
\begin{equation}\label{eq:PN_def}
\PN{\Tu}(\x)
=
\sum_{\k \in \ZNd} 
\FT{\Tu}(\k)
\bfun{\k}( \x )
\text{ for }
\x \in \puc,
\end{equation}
is an orthogonal projection in scalar product on $\Lper{2}{}$ and also on $\Hper{s}{\xRd}$ for any
$s \geq 0$. Second, the interpolation operator $\QN{} : \Cper{0}{\xRd}
\rightarrow \TrigN$, whose action is expressed as
\begin{equation}\label{eq:QN_def}
\QN{\Tu}( \x ) 
=
\sum_{\k \in \ZNd}
\Tu( \xk )
\bfunN{\k}(\x)
\text{ for }
\x \in \puc.
\end{equation}
Note that this operator is a projection, but no longer an orthogonal one. The
following lemma, proven in \Aref{app:trigonometric}, summarizes
the approximation properties of both operators.

\begin{lemma}
\label{lem:approximation}
Let \eqref{eq:A4} hold. Then for $\Tu \in \Lper{2}{\xRd}$ 
\begin{align}\label{eq:conv_PN}
\lim_{\TN \rightarrow \T{\infty}}
\norm{\Tu - \PN{\Tu}}{%
\Lper{2}{\xRd}}  
=
0,
\end{align}
where $\TN \rightarrow \T{\infty}$ stands for $\min_\alp N_\alp \rightarrow
\infty$. Next, for $\Tu \in \Hper{s}{\xRd}$ with $s \geq r \geq 0$
\begin{align}\label{eq:rate_of_conv_PN}
\norm{\Tu - \PN{\Tu}}{%
\Hper{r}{\xRd}}  
\leq
\Ch^{s-r}
\norm{\Tu}{\Hper{s}{\xRd}},
\end{align}
and for $s > \dime/2$
\begin{align}\label{eq:rate_of_conv_QN}
\norm{\Tu - \QN{\Tu}}{%
\Hper{r}{\xRd}}
\leq
c_{r,s}
\Ch^{s-r}
\norm{\Tu}{\Hper{s}{\xRd}},
\end{align}
with
\begin{align*}
c_{r,s}^2 
= 
1 + \dime^{r} \rh^{2r} 
\sum_{\m \in \NOmO}
\norm{\m}{\xRd}^{-2s}.
\end{align*}
\end{lemma}

An essential advantage of trigonometrical polynomials is
that they, under assumption~\eqref{eq:A4}, allow us to construct
structure-preserving conforming finite-dimensional approximations of spaces
$\U$, $\E$ and $\J$ in a transparent way. This is simply done by
setting
\begin{align*}
\UN = \U \cap \TrigN = \PN{\U}, 
\,
\EN = \E \cap \TrigN = \PN{\E}, 
\,
\JN = \J \cap \TrigN = \PN{\J},
\end{align*}
where $\UN$, $\EN$, $\JN$ collect the constant, zero-mean curl- and
divergence-free trigonometric polynomials, respectively. Moreover, since
$\PN{}$ is an orthogonal projection from $\Lper{2}{\xRd}$ to $\TrigN$, a
``trigonometric'' variant of the Helmholtz decomposition~\eqref{eq:Helmholtz} holds:
\begin{align}\label{eq:trig_Helmholtz}
\TrigN = \UN \oplus \EN \oplus \JN.
\end{align}

\subsection{Galerkin approximation}\label{sec:conforming_Galerkin}
Having specified the finite-dimensional spaces we will work with, we now
follow the standard route to discretize the problem by the Galerkin method. The
corresponding notion of approximate solution and its qualitative properties
follow next.

\begin{definition}
A field $\TGE \in \EN$ is a solution to the Galerkin
approximation of the cell problem~\eqref{eq:weak_solution} if
\begin{align}\tag{Ga}
\bform{\TGE}{\TvN} = \lform{\TvN}
\text{ for all }
\TvN \in \EN.
\end{align}
\end{definition}
\begin{proposition} 
Let~\eqref{eq:A1},~\eqref{eq:A2} and~\eqref{eq:A4} hold.
Then, there is the unique $\TGE$ satisfying
\begin{align*}
\lim_{\TN \rightarrow \T{\infty}}
\norm{\Tefl - \TGE}{\Lper{2}{\xRd}}
=
0.
\end{align*}
If, in addition, $\Tefl \in \Hper{s}{\xRd}$ for $s > 0$, we have
\begin{align}\label{eq:Ga_rate_of_convergence}
\norm{\Tefl-\TGE}{\Lper{2}{\xRd}}
\leq 
\rA \Ch^s \norm{\Tefl}{\Hper{s}{\xRd}}.
\end{align}
\end{proposition}

\begin{proof}
Existence and uniqueness of the solution is a direct consequence of the
Lax-Milgram lemma, as the estimates~\eqref{eq:form_estimates} still hold. From
the C\'{e}a lemma~\cite{Cea:1964:AVP}, we infer that
$$
\norm{
\Tefl
-
\TGE
}{\Lper{2}{\xRd}}
\leq 
\rA 
\inf_{\TvN \in \EN}
\norm{\Tefl - \TvN}{%
\Lper{2}{\xRd}}
\leq
\rA 
\norm{\Tefl - \PN{\Tefl}}{%
\Lper{2}{\xRd}}.
$$
The statement of the proposition now follows from estimates~\eqref{eq:conv_PN}
and~\eqref{eq:rate_of_conv_PN}.
\end{proof}

\subsection{Galerkin approximation with numerical
integration}\label{sec:non_conforming_Galerkin}

The discretization procedure introduced in the previous section rests on the
assumption that the linear and bilinear forms are evaluated exactly. Of
course, this can only be made for specific forms of coefficients $\TA$, see
also \Sref{sec:Brisard} for further discussion, and in the general case a
numerical integration needs to be employed. For the trigonometric
polynomial-based discretization, the natural choice is to employ the
interpolation operator $\QN{}$ and perform the integration by utilizing
relation~\eqref{eq:scal_trig_pol}. This results in parameter-dependent
forms $a_\TN  : \TrigN \times \TrigN \rightarrow \R$ and~$b_\TN : \TrigN
\rightarrow \R$ given by
\begin{subequations}\label{eq:formsN_def}
\begin{align}
\bformN{\TuN}{\TvN}
&=
\sprod{\QN{\TA\TuN}}{\TvN}{\Lper{2}{\xRd}},
\\
\lformN{\TvN}
&=
-
\sprod{\QN{\TA\TE}}{\TvN}{\Lper{2}{\xRd}},
\end{align}
\end{subequations}
and a computable solution specified next. 
\begin{definition}
A field $\TGEi \in \EN$ is the solution to the Galerkin approximation of the
cell problem~\eqref{eq:weak_solution} with numerical integration if
\begin{align}\tag{GaNi}\label{eq:GaNi}
\bformN{\TGEi}{\TvN} 
=
\lformN{\TvN}
\text{ for all }
\TvN \in \EN.
\end{align}
\end{definition}

Due to involvement of the interpolation operator $\QN{}$, data of the problem
must satisfy
\begin{align}\tag{A1'}\label{eq:A1a}
\TA
\in 
\Cper{0}{\xRdd}
\end{align}
to ensure that the forms~\eqref{eq:formsN_def} are well-defined, and 
\begin{align}\tag{A1$^\star$}\label{eq:A1b}
\TA \in \Wper{s}{\infty}{\xRdd}
\text{ with } s > \dime / 2,
\end{align}
to estimate the rate of convergence in an analogous way to~\eqref{eq:Ga_rate_of_convergence}. 

\begin{proposition}\label{prop:GaNi}
Let \eqref{eq:A2} and~\eqref{eq:A4} hold. Then, under
assumption~\eqref{eq:A1a}, there is the unique solution $\TGEi$. If, in
addition, \eqref{eq:A1b} is satisfied and $\Tefl \in \Hper{s}{\xRd}$,
\begin{align*}
\norm{\Tefl-\TGEi}{\Lper{2}{\xRd}}
\leq & 
C\Ch^s,
\end{align*}
with the discretization-independent constant given by
\begin{align*}
 C =  
\bigl(\rA+1\bigr) \norm{\Tefl}{\Hper{s}{\xRd}}
+
\frac{c_{0,s}}{\cA} \norm{\TA}{\Wper{s}{\infty}{\xRdd}} \left(
\norm{\Tefl}{\Hper{s}{\xRd}} + \norm{\TE}{\xRd} \right).
\end{align*}
\end{proposition}

\begin{proof} 
Existence and uniqueness of the solution rely again on the Lax-Milgram
lemma, once observing that the estimates~\eqref{eq:form_estimates} hold also for
$a_\TN$ and $b_\TN$. For example, to verify the coercivity of the bilinear form,
consider $\TuN \in \TrigN$, and combine~\eqref{eq:scal_trig_pol}
and~\eqref{eq:A2} to obtain
\begin{align*}
\bformN{\TuN}{\TuN} 
& =
\sprod{\QN{\TA \TuN}}{\TuN}{\Lper{2}{\xRd}}
\\
&=
\frac{1}{|\TN|}
\sum_{\k \in \ZNd}
\sprod{\TA(\xk) \TuN(\xk)}{\TuN(\xk)}{\xRd}
\\
& \geq
\frac{\cA}{|\TN|}
\sum_{\k \in \ZNd}
\sprod{\TuN(\xk)}{\TuN(\xk)}{\xRd}
=
\cA
\norm{\TuN}{\Lper{2}{\xRd}}^2.
\end{align*}
The remaining estimates are established by similar arguments. 

As a consequence of the second Strang lemma~\cite{Strang:1972:VCF}, we have
\begin{align*}
\norm{\Tefl -\TGEi}{\Lper{2}{\xRd}}
\leq & 
\frac{1}{\cA}
\sup_{\TuN \in \EN}
\frac{| \lform{\TuN} - \lformN{\TuN} |}%
{\norm{\TuN}{\Lper{2}{\xRd}}}
\\
& +
\inf_{\TvN \in \EN}
\Bigl[
(1 + \rA )
\norm{\Tefl - \TvN}{\Lper{2}{\xRd}}
\\
& + 
\frac{1}{\cA}
\sup_{\TuN \in \EN}
\frac{| \bform{\TvN}{\TuN} - \bformN{\TvN}{\TuN} |}%
{\norm{\TuN}{\Lper{2}{\xRd}}}
\Bigr].
\end{align*}
We estimate the differences between the forms by the Cauchy-Schwartz inequality
\begin{align*}
| \lform{\TuN} - \lformN{\TuN} |
& \leq
\norm{\TA\TE  - \QN{\TA\TE}}{\Lper{2}{\xRd}}
\norm{\TuN}{\Lper{2}{\xRd}},
\\
| \bform{\TvN}{\TuN} - \bformN{\TvN}{\TuN} |
& \leq
\norm{\TA \TvN - \QN{\TA\TvN} }{\Lper{2}{\xRd}}
\norm{\TuN}{\Lper{2}{\xRd}},
\end{align*}
and set $\TvN = \PN{\Tefl}$. Now, relations \eqref{eq:rate_of_conv_PN} and
\eqref{eq:rate_of_conv_QN} with $r=0$ yield
\begin{align*}
\norm{\Tefl - \PN{\Tefl}}{\Lper{2}{\xRd}}
& \leq 
\Ch^s 
\norm{\Tefl}{\Hper{s}{\xRd}},
\\
\norm{\TA \PN{\Tefl} - \QN{\TA\PN{\Tefl}} }{\Lper{2}{\xRd}}
& \leq
c_{0,s} \Ch^s 
\norm{\TA \PN{\Tefl}}{\Hper{s}{\xRdd}}
\\
& \leq 
c_{0,s} \Ch^s 
\norm{\TA}{\Wper{s}{\infty}{\xRdd}}
\norm{\Tefl}{\Hper{s}{\xRd}},
\end{align*}
where in the last inequality, we used the fact that the truncation operator  $\PN{}$ is
the orthogonal projection from $\Hper{s}{\xRd}$ onto $\TrigN$. The estimate
\begin{align*}
\norm{\TA\TE  - \QN{\TA\TE} }{\Lper{2}{\xRd}}
\leq
c_{0,s} \Ch^s 
\norm{\TA}{\Wper{s}{\infty}{\xRdd}}
\norm{\TE}{\xRd}
\end{align*}
is established by analogous arguments. Utilizing these estimates in
the Second Strang lemma completes the proof.
\end{proof}

To close this section, let us note
that the proofs of the rate of convergence of approximate solutions require
sufficient regularity of the weak solution, i.e. $\Tefl \in \Hper{s}{\xRd}$.
In the context of \Sref{sec:non_conforming_Galerkin}, these assumptions
are not too restrictive, since $\TA \in \Wper{s}{\infty}{\xRdd}$ with $s \in
\set{N}$ implies that $\Tefl \in \Hper{s}{\R^\dime}$.
A short proof of this result is given in
\Aref{app:regularity}, in order to make the paper self-contained. In addition,
computational experiments supporting the statements of \Pref{prop:GaNi} are
available in~\cite[pp.~142--145]{Vondrejc:2013:FFT}.

\nomenclature{$\Tv$}{Test function}%

\section{Algebraic system and its solution}
\label{sec:Algebraic_system}
The present section is dedicated to the analysis of the fully discrete version
of~Galerkin approximation with Numerical integration \eqref{eq:GaNi}. In view of the representation formula for trigonometric
polynomials~\eqref{eq:trig_fourier}, we find it convenient to base our
approach on structured vectors and matrices storing the values at the grid
points both in the real and in the Fourier domains. Therefore, the goal of
\Sref{sec:discrete_notation} is to refine the notation and adapt the
relevant results to the discrete case. The notion of the discrete solution is
presented in \Sref{sec:fully_discrete}. Here, we also show its equivalence to
linear systems arising from the discrete Lippmann-Schwinger equation and from
the variational formulation. In \Sref{sec:CG}, we demonstrate that the latter
system is solvable by the Conjugate gradient algorithm and provide its
connection to the original Moulinec-Suquet scheme~\cite{Moulinec:1994:FNMC},
thereby explaining our earlier computational
observations~\cite{ZeVoNoMa:2010:AFFTH}.

\subsection{Notation and preliminaries}\label{sec:discrete_notation}

A multi-index notation is systematically employed, in which $\X^{\TN}$
represents $\X^{N_1 \times \cdots \times N_\dime}$. Then the sets $\xXN$ and
$\xMN$, or their complex counterparts, represent the space of structured vectors
and matrices denoted by bold serif font, e.g. $\MBv =
(v_{\alp}^\k)_{\alp}^{\k\in\ZN} \in \xXN$ and $\MB{M} =
(M_{\alp\bet}^{\k\m})_{\alp,\bet}^{\k,\m\in\ZN} \in \xMN$; sub-vectors and
sub-matrices are designated by superscripts, e.g. $\MBv^\k =
(v_{\alp}^\k)_{\alp} \in \xRd$ or $\MB{M}^{\k\m} =
(M_{\alp\bet}^{\k\m})_{\alp,\bet} \in \xRdd$. The scalar product on e.g. $\xXN$
is hence defined as
\begin{align*}
\sprod{\MB{u}}{\MBv}{\xXN} 
= 
\frac{1}{\meas{\TN}}
\sum_{\k\in\ZNd} 
\sprod{\MBu^{\k}}{\MBv^{\k}}{\xRd},
\end{align*}
and the structured matrix-vector or matrix-matrix multiplications
follow from
\begin{align*}
(\MB{M}\MBv)^{\k} 
= 
\sum_{\m\in\ZNd}
\MB{M}^{\k\m}
\MBv^{\m}
\in \xRd
\text{ or }
(\MB{M} \MB{L})^{\k\m} 
=
\sum_{\n\in\ZNd}
\MB{M}^{\k\n}\MB{L}^{\n\m}
\in \xRdd,
\end{align*}
for $\k,\m \in \ZNd$ and $\MB{L} \in \xMN$. For later purposes, we also
collect input data in the form of structured matrices and vectors
\begin{align*}
\MBA_\TN
&=
\Bigl(
\del_{\k\m}A_{\alp\beta}(\xk)
\Bigr)_{\alp,\beta}^{\k,\m\in\ZNd},
&
\MBE_\TN
&=
\Bigl( E_{\alp}
\Bigr)_{\alp}^{\k\in\ZNd}, 
&
\MBA_\TN\refe
&=
\Bigl(
\del_{\k\m}A\refe_{\alp\beta}
\Bigr)_{\alp,\beta}^{\k,\m\in\ZNd}, 
\end{align*}
assuming that~\eqref{eq:A1a} holds.

The relation between $\TrigN$ and $\xRdN$ is established by an operator
transforming the values at the grid points into a structured vector, i.e.
\begin{align*}
\IN{} : \Cper{0}{\xRd} \rightarrow \xRdN,
&&
\IN{\Tu_\TN} = \left(\Tu_\TN(\xk) \right)^{\k \in
\ZNd} \in \xRdN.
\end{align*}
The following lemma summarizes its properties and applications.

\begin{lemma}\label{lemma:isometry}
Under~\eqref{eq:A4}, 
the operator $\IN{}$ is an one-to-one isometric map from $\TrigN$ onto $\xRdN$,
i.e. for all $\Tu_\TN,\Tv_\TN \in \TrigN$
\begin{align*}
\sprod{\Tu_\TN}{\Tv_\TN}{\Lper{2}{\xRd}}
=
\sprod{\IN{\Tu_\TN}}{\IN{\Tv_\TN}}{\xRdN}.
\end{align*}
Hence, under~\eqref{eq:A1a}, 
\begin{align*}
\bformN{\TuN}{\TvN} 
=
\sprod{\MBA_\TN\MBu_\TN}{\MBv_\TN}{\xRdN}, 
&&
\lformN{\Tv_\TN} 
= 
-\sprod{\MBA_\TN\MBE_\TN}{\MBv_\TN}{\xRdN},
\end{align*}
with $\MBu_\TN = \IN{\TuN}$ and $\MBv_\TN = \IN{\TvN}$.
\end{lemma}

\begin{proof}
Both statements are consequences of basic properties of fundamental
trigonometric polynomials. Indeed, from \eqref{eq:trig_real} we see that every
trigonometric polynomial is uniquely defined by its grid values and, for a given
$\Tu_\TN$ and $\Tv_\TN \in \TrigN$, we infer
\begin{align*}
\sprod{\Tu_\TN}{\Tv_\TN}{\Lper{2}{\xRd}}
& = 
\sum_{\k,\m\in\ZNd} 
\sprod{\Tu_\TN(\xk)}{\Tv_\TN(\xm)}{\xRd}
\cdot
\sprod{\bfunN{\k}}{\bfunN{\m}}{\Lper{2}{}}
\\
& = 
\sum_{\k,\m\in\ZNd} 
\sprod{\Tu_\TN(\xk)}{\Tv_\TN(\xm)}{\xRd}
\cdot 
\frac{\del_{\k\m}}{\meas{\TN}}
\\
&=
\sprod{\IN{\Tu_\TN}}{\IN{\Tv_\TN}}{\xRdN}.
\end{align*}
The discrete representation of the bilinear form $a_\TN$ follows from
\begin{equation*}
\bformN{\TuN}{\TvN} 
=
\sprod{\QN{\TA\Tu_\TN}}{\Tv_\TN}{\Lper{2}{\xRd}}
= 
\sprod{\IN{\QN{\TA\Tu_\TN}}}{\IN{\TvN}}{\xRdN},
\end{equation*}
and from the fact that $\IN{\QN{\Tf}} = \IN{\Tf}$ for any $\Tf \in \Cper{0}{\xRd}$. The
representation of the linear form $b_\TN$ is established in the same way.
\end{proof}

With the help of operator $\IN{}$, discrete analogues of
sub-spaces~\eqref{eq:trig_Helmholtz} are simply provided by\footnote{%
Note that the space $\xEN$ should not be mistaken with the structured vector
$\MBE_\TN\in\xXN$ storing the grid values of the average gradient field $\TE$.}
\begin{align*}
\xUN &= \IN{\UN},
&
\xEN &= \IN{\EN},
&
\xJN &= \IN{\JN},
\end{align*}
where $\xUN$, $\xEN$, and $\xJN$ collect the grid values of constant,
zero-mean curl- and divergence-free trigonometric polynomials
with values in $\xRd$.
It also follows from \Lref{lemma:isometry} that the Helmholtz decomposition
property~\eqref{eq:Helmholtz} is inherited also in the discrete setting, i.e.
\begin{align}\label{eq:Helmholtz_discrete}
\xRdN = \xUN \oplusp \xEN \oplusp \xJN.
\end{align}

To provide projection operators to $\xEN$, we proceed in the same way as in
\Sref{sec:projection_operator}. First, using~\eqref{eq:Gamma_hat}, we
represent the Fourier transform of the kernel of the Lippmann-Schwinger equation
as
\begin{align*}
 \FTMBGamm 
 = 
 \left(\del_{\k\m} \ThGref(\k)\right)^{\k,\m\in\ZNd}\in\xMN,
\end{align*}
and transform it to the real space by means of matrices
%
\begin{align*}
\DFT &= 
\frac{1}{\meas{\TN}}
\left(\del_{\alp\beta}\omega_{\TN}^{-\k\m}\right)_{\alp,\beta}^{\k,\m\in\ZNd}\in\xhMN, 
&
\iDFT &= \left( \del_{\alp\beta}\omega_{\TN}^{\k\m}
\right)_{\alp,\beta}^{\k,\m\in\ZNd} \in\xhMN,
\end{align*}
%
implementing the forward and inverse discrete Fourier transforms,
recall~\eqref{eq:DFT_def}. This results in
\begin{align}\label{eq:FFT_matrices}
\MBGamm = \iDFT \FTMBGamm \DFT \in \xMN, 
&&
\OGD = \MBGamm \MBA_\TN\refe \in  \xMN;
\end{align}
we also set $\OGDo = \OGD$ for $\TA\refe = \T{I}\in\xRdd$. By translating
\Lref{lem:projection} to the current representation, we obtain:

\begin{lemma}\label{lem:FD_projection}
Let~\eqref{eq:A3} and~\eqref{eq:A4} be satisfied. Then, the following statements
hold:
\begin{itemize}
  \item[(i)] $\OGD$ is a well-defined bounded operator 
  $\xRdN \rightarrow \xRdN$,
  \item[(ii)] the adjoint operator to $\OGD$ is given by $\OGDadj
  \MBA\refe_\TN = \MBA\refe_\TN \OGD$,
  \item[(iii)] $\OGD$ is a projection onto $\xEN$,
  \item[(iv)]  $\OGD[\MBu] = \T{0}$ for all $\MBu \in \xUN$ and 
  $\sprod{\OGD[\MBu]}{\MBv}{\xRdN}
  = 0$ for all $\MBu \in \xRdN$ and $\MBv \in \xJN \oplusp
  \xUN$,
  \item[(v)] for $\TAref = \lambda\T{I}$
  with $\lambda > 0$, $\OGD$ becomes an orthogonal projection $\OGDo$ independent of $\TAref$. 
\end{itemize}
\end{lemma}

\begin{proof}
The proof follows exactly the same route as for \Lref{lem:projection}.
\end{proof}

To complete our exposition, we highlight close connections among the spaces
involved in the discretization of~\eqref{eq:weak_solution}. This is
schematically shown in the following diagram, which, under
assumption~\eqref{eq:A4}, commutes.

\begin{figure}[ht]\label{fig:scheme}
\begin{center}
\includegraphics[width=.7\textwidth]{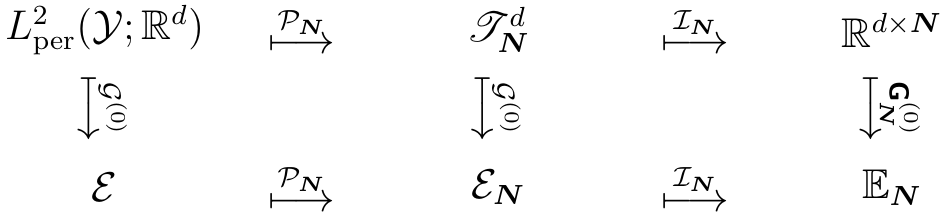}
\end{center}
\caption{Discretization strategy.}
\end{figure}

\subsection{Fully discrete formulations}\label{sec:fully_discrete}
After introducing the general concepts in the previous section, now we are ready
to convert the variational problem~\eqref{eq:GaNi} into its fully discrete
version. 

\begin{definition} 
A structured vector $\MBefl_\TN \in \xEN$ is a solution to the fully discrete
form of~\eqref{eq:GaNi} if
\begin{align}\tag{GaNiD}\label{eq:GaNiD}
\sprod{\MBA_\TN\MBefl_\TN}{\MBv_\TN}{\xRdN} 
=
-\sprod{\MBA_\TN\MBE_\TN}{\MBv_\TN}{\xRdN}
\text{ for all }
\MBv_\TN\in\xEN.
\end{align}
Moreover, a solution to the discrete Lippmann-Schwinger equation $\MBe_\TN \in
\xRdN$ satisfies
\begin{align}\tag{L-SD}\label{eq:LSD}
[ \MBI + \MBGamm (\MBA_\TN - \MBA^{(0)}) ] \MBe_\TN 
=
\MBE_\TN. 
\end{align}
\end{definition}

The following result shows that, as expected, these solutions coincide and can
also be related to an equivalent system of linear equations.

\begin{proposition}
Let \eqref{eq:A1a} and \eqref{eq:A2}--\eqref{eq:A4} be satisfied. Then, the
following holds: 
\begin{itemize}
  \item[(i)] the unique solution to~\eqref{eq:GaNiD} is given by
  $\MBefl_\TN = \IN{\Tefl_\TN}$,
  \item[(ii)] the unique solution to~\eqref{eq:LSD} satisfies $\MBe_\TN =
  \MBE_\TN + \MBefl_\TN$,
  \item[(iii)] for $\TAref = \lambda \T{I}$, with $\lambda >
  0$,~\eqref{eq:GaNiD} is equivalent to the linear system for $\MBefl_\TN\in\xEN$
\begin{align}\label{eq:final_system}
\OGDo \MBA_\TN \MBefl_\TN =  -\OGDo \MBA_\TN \MBE_\TN.
\end{align}
\end{itemize}
\end{proposition}

\begin{proof}
As for~(i), it is an easy consequence of the properties of operator $\IN{}$
demonstrated in~\Lref{lemma:isometry}. (ii)~proceeds in the same way as in the
continuous case, recall the proof of \Pref{prop:equivalence}. Finally, from
\Lref{lem:FD_projection}(iii) we infer that all test matrices $\MBv_\TN$ can be
expressed as $\OGDo \MBu_\TN$ with $\MBu_\TN \in \xRdN$, see also
\Fref{fig:scheme}. The statement~(iii) thus follows from the self-adjointness of
$\OGDo$, cf. \Lref{lem:FD_projection}(v).
\end{proof}

Let us note that $\MBe_\TN$ can be interpreted as grid values of the solution to
yet another discretization of the Lippmann-Schwinger equation via the
trigonometric collocation method, which consists of
projecting~\eqref{eq:integral_solution} to the space of trigonometric polynomials
by operator $\QN{}$ for sufficiently regular data. An interested reader is
referred to~\cite{Vainikko:2000:FSLS} for the general setup and
to~\cite{ZeVoNoMa:2010:AFFTH} for specific application to periodic homogenization
problems.

\subsection{Solution of linear system}\label{sec:CG}
A closer inspection reveals that the non-symmetric matrix 
in~\eqref{eq:final_system} is a product of sparse structured matrices, and that
the cost of its action is governed by the contributions of $\DFT$ and $\iDFT$,
recall~\eqref{eq:FFT_matrices}. Since this step can by performed by the Fast
Fourier Transform techniques~\cite{Cooley:1965:AMC} in $O(\meas{\TN} \log
\meas{\TN})$ operations, the system~\eqref{eq:final_system} can be efficiently
solved by iterative methods. In fact, the next lemma clarifies that the standard
conjugate gradient algorithm~\cite{Hestenes:1952:MCG} works well, even though
the system matrix is non-symmetric.

\begin{lemma}
System~\eqref{eq:final_system} can be solved by the conjugate gradient algorithm
for an arbitrary initial solution $\MBefl_{\TN,(0)} \in \xEN$.
\end{lemma}

\begin{proof}
The proof relies on the fact that the conjugate gradient algorithm is a
special instance of the orthogonal projection method for symmetric and
positive-definite system matrices~\cite[Section~6.7]{Saad:2003:IMSL}. To this
goal, define the $i$-th Krylov subspace as
\begin{align*}
\set{K}_{(i)} 
=
\mathrm{span}
\Bigl\{ 
  \MB{r}_{\TN,(0)}, 
  \OGDo \MBA_\TN \MB{r}_{\TN,(0)}, 
  \ldots, 
  (\OGDo\MBA_\TN)^{i-1} \MB{r}_{\TN,(0)} 
\Bigr\},
\end{align*}
where the residual vector corresponding to the initial guess is given
by
\begin{align*}
\MB{r}_{(0)} 
= 
\OGDo \MBA_\TN 
\left( \MBefl_{\TN,(0)} + \MBE_\TN \right).
\end{align*}
%

Due to involvement of structured matrix $\OGDo$, inclusions
$\set{K}_{(i)}\subset \xEN$
 hold for all $i$, cf.~\Lref{lem:FD_projection}(iii).
The $i$-th iterate of the orthogonal projection method is searched in the form
$\MBefl_{\TN,(i)} = \MBefl_{\TN,(0)} + \MBu_{\TN,(i)}$, with
$\MBu_{\TN,(i)} \in \set{K}_{(i)}$ satisfying, e.g.~\cite[Section~6.4]{Saad:2003:IMSL},
\begin{align*}
\sprod{%
\OGDo \MBA_\TN\MBu_{\TN,(i)}
}{\MBv_\TN}{\xRdN}
= 
-
\sprod{%
\OGDo \MBA_\TN (\MBE_\TN + \MBefl_{\TN,(0)})
}{\MBv_\TN}{\xRdN}
\text{ for all }
\MBv_\TN \in \set{K}_{(i)}.
\end{align*}
Due to self-adjointness of $\OGDo$ and its $\xEN$-invariance,
\Lref{lem:FD_projection}(iii) and~(v), this is equivalent to 
\begin{align}\label{eq:CG_identity}
\sprod{%
\MBA_\TN\MBu_{\TN,(i)}
}{\MBv_\TN}{\xRdN}
= 
-
\sprod{%
\MBA_\TN (\MBE_\TN + \MBefl_{\TN,(0)})
}{\MBv_\TN}{\xRdN}
\text{ for all }
\MBv_\TN \in \set{K}_{(i)}.
\end{align}
As $\MBA_\TN$ is symmetric and positive-definite, the previous relation
represents convergent iterations of the conjugate gradient
method~\cite[Section~6.7]{Saad:2003:IMSL}. Moreover, since $\MBefl_{\TN,(0)}
\in \xEN$, all iterates $\MBefl_{\TN,(i)}$ remain in $\xEN$.
\end{proof}

Several comments are now in order to clarify the relevance of the presented
results to the original Moulinec-Suquet scheme~\cite{Moulinec:1994:FNMC} and to
our computational experiments~\cite{ZeVoNoMa:2010:AFFTH}, both related to the
discrete Lippmann-Schwinger equation~\eqref{eq:LSD}. First note that the
Moulinec-Suquet method consists of solving~\eqref{eq:LSD} with the Neumann
series expansion
\begin{align*}  
\MBe_{\TN,(i)}
=
\sum_{j=0}^{i}
\left( - \MBGamm (\MBA_\TN - \MBA\refe) \right)^j \MBE_\TN
=
-
\MBGamm (\MBA_\TN - \MBA\refe)
\MBe_{\TN,(i-1)}
+
\MBE_\TN,
\end{align*}
the convergence of which depends on the choice of $\TAref$ and the number of
iterators needed to reach a given tolerance increases linearly with the contrast
in coefficients $\rA$ since, for the optimal choice of $\TAref$, 
\begin{align*}
\norm{\MBe_{\TN,(i)} - \MBe_\TN}{\xRdN}
\leq 
C 
\left( \frac{\rA-1}{\rA+1} \right)^i
\norm{\MBe_{\TN,(i)} - \MBe_{\TN,(0)}}{\xRdN},
\end{align*}
see e.g.~\cite{Eyre:1999:FNS} or \cite[Section~4.2.1]{Saad:2003:IMSL}.

Second, it follows from the previous proof that system~\eqref{eq:LSD} can be
solved by the conjugate gradient method for any $\TAref = \lambda \T{I}$, since
$\MBu_{\TN,(i)}$ in~\eqref{eq:CG_identity} can be transferred to the equivalent
solution of discrete Lippmann-Schwinger equation on $\set{K}_{(i)}$,
repeating in verbatim the proof of \Pref{prop:equivalence},
cf.~\cite{VoZeMa:2012:LNSC}.

Third, when either of
systems~\eqref{eq:final_system} or~\eqref{eq:LSD} is resolved by the conjugate
gradient method, the number of iterations needed for a given tolerance grows
as $\sqrt{\rA}$. This follows from the fact that, due to involvement of
the projection $\MBG_\TN$, the iterates never leave $\xEN$ and the condition
number of the system matrix in~\eqref{eq:CG_identity}  satisfies
$\kappa(\MBA_\TN) = \rA$. Thus, a well-known result of the convergence
analysis of the conjugate gradient method,
e.g.~\cite[Section~6.11.3]{Saad:2003:IMSL}, implies that  
\begin{align*}
\norm{\MBefl_{\TN,(i)} - \MBefl_\TN}{\xRdN}
\leq 
C 
\left( \frac{\sqrt{\rA}-1}{\sqrt{\rA}+1} \right)^i
\norm{\MBefl_{\TN,(i)} - \MBefl_{\TN,(0)}}{\xRdN},
\end{align*}
see also~\cite[Section~3.2]{ZeVoNoMa:2010:AFFTH} for further discussion.

Finally, we wish to emphasize that the results of the present section are
supported by simulation performed in two~\cite{ZeVoNoMa:2010:AFFTH} and
three~\cite{VoZeMa:2012:LNSC} dimensions.

\section{Conclusions}\label{sec:conclusions}
%
In this paper, we have introduced a Galerkin framework for the discretization of
the cell problem arising in periodic homogenization theories. Our approach
builds on a finite-dimensional approximation space formed by trigonometric
polynomials, and on a suitable projection operator reflecting the differential
constraints in the problem formulation. In the scalar elliptic setting, we have
demonstrated that

\begin{itemize}
  \item trigonometric polynomials provide a transparent way to
  constructing conformal structure-preserving approximations to
  infinite-dimensional curl-free spaces~\eqref{eq:E_def},
  \item solutions to the discretized problems~(with or without numerical
  integration) converge to the weak solution, with standard rates of
  convergence for sufficiently regular data,
  \item the Galerkin method with numerical integration~\eqref{eq:GaNi} is
  equivalent to the discrete Lippmann-Schwinger equation forming the basis of
  the original\linebreak Moulinec-Suquet scheme~\cite{Moulinec:1994:FNMC},
  \item the non-symmetric linear system arising from~\eqref{eq:GaNi} is
  independent of the auxiliary parameter $\TAref$ and can be solved by the
  conjugate gradient method.
\end{itemize} 

Apart from completely explaining our earlier
observations~\cite{ZeVoNoMa:2010:AFFTH}, we believe that the presented
results provide a convenient starting point for several interesting extensions.
First, utilizing the Helmholtz decomposition~\eqref{eq:Helmholtz} and its
discrete variant~\eqref{eq:Helmholtz_discrete}, the dual formulation
of~\eqref{eq:weak_solution} can be solved in completely analogous way to provide
computable and reliable a-posteriori error estimates for the approximate
solutions. Such results have already been announced in the Ph.D. thesis of the first author, cf.~\cite[pp.~121--148]{Vondrejc:2013:FFT}, and independently in \cite{VoZeMa:2013:GarBounds}. Second, since our approach relies on the well-established concept
of the weak solution, it might provide a unifying basis to establish connections
among various refinements of the original scheme briefly discussed in
\Sref{sec:Introduction}. Third, performance of e.g.
multi-grid~\cite{Eyre:1999:FNS} or stochastic~\cite{Xu:2005:SCM} solvers can be
significantly improved by variational techniques developed in this work.
Finally, we may proceed beyond the scalar setting to more complex physical
phenomena, or to modeling of real-world material systems. We plan to explore
some of these possibilities in future investigations.

\section{Comparison with results by Brisard and
Dormieux~\cite{Brisard:2012:CGA}}\label{sec:Brisard}

As already stated in the introductory section, this is not the first paper to
interpret the Moulinec-Suquet method as a Galerkin scheme. To the best of our
knowledge, such connection was first made by Brisard and
Dormieux~\cite{Brisard:2010:FFT} in 2010 for linear elasticity, and was later
refined by convergence analysis~\cite{Brisard:2012:CGA}. Here we briefly
comment on the differences between their developments and the results
presented here.

The approach taken by the authors
of~\cite{Brisard:2010:FFT,Brisard:2012:CGA} proceeds from the discretization of
stationarity conditions to the Hashin-Shtrikman
functional~\cite{Hashin:1962:SMP}, expressed in terms of an unknown polarization
field $\T{\tau} \in \Lper{2}{\xRd}$ as
\begin{align}\label{eq:HS_stat}
\sprod{(\TA - \TAref)^{-1}\T{\tau}}{\Tv}{\Lper{2}{\xRd}} 
+
\sprod{\T{\Gamma}_0\T{\tau}}{\Tv}{\Lper{2}{\xRd}} 
=
\scal{\TE}{\Tv}_{\Lper{2}{\xRd}}
\end{align}
for all $\Tv\in\Lper{2}{\xRd}$, which is equivalent to the Lippmann-Schwinger
equation~\eqref{eq:integral_solution} with $\T{\tau} = (\TA-\TA\refe)\Te$. Since
the polarization field is sought in the whole space $\Lper{2}{\xRd}$, instead of
the subspace $\E$ of zero mean curl-free functions as
in~\eqref{eq:weak_solution}, the approximation space consists of pixel- or
voxel-wise constant fields and~\eqref{eq:HS_stat} can be localized to individual
pixels/voxels. Therefore, the crucial step consists in the evaluation of the
term
\begin{align}
\label{eq:B-D_term}
\int_{\puc}
\TGref(\x - \y)
\T{\tau}(\y)
\de \y
=
\sum_{\k \in \Zd}
\ThGref(\k)
\FT{\T{\tau}}(\k)
\bfun{\k}(\x),
\end{align}
representing the negative value of the fluctuating gradient field $-\Tefl \in
\E$, cf.~\eqref{eq:Lippman-Schwinger} and~\eqref{eq:Gamma_action}.

Similarly to our work, two different approximations are considered. The first
one relies on the so-called consistent Green operator, for which the sum
in~\eqref{eq:B-D_term} is computed exactly,
cf.~\Sref{sec:conforming_Galerkin}. In the non-consistent case, the infinite sum
is truncated to $\k \in \ZNd$, with an effect comparable to the numerical
integration in \Sref{sec:non_conforming_Galerkin}. Convergence of the
approximate solutions is proven in an analogous manner to the present
work~\cite{Brisard:2012:CGA}, but no a-priori estimates on the rate of
convergence are provided.

Albeit the underlying ideas and mathematical instruments used in both approaches
are similar, they lead to different schemes. In particular, in order to
employ the consistent Green operator, one needs to evaluate the lattice sums
in~\eqref{eq:B-D_term} to a high accuracy, which is rather difficult (especially
in the three-dimensional setting). On the other hand, truncating the sum
in~\eqref{eq:B-D_term} generates errors arising from the numerical integration,
and produces non-conforming gradient fields $-\Tefl \not \in \E$, which implies
that the discrete Helmholtz decomposition property~\eqref{eq:trig_Helmholtz} is
no longer valid. Finally, since the actual status of the stationary point
in~\eqref{eq:HS_stat}, i.e. minimizer, maximizer or saddle point, depends
on the choice of $\TAref$, the matrix of the resulting system of linear
equations till depends on $\TAref$ and can be either positive-definite,
negative-definite or indefinite. As a result, more complex iterative solvers
need to be employed~\cite{Brisard:2012:CGA}.

\appendix
\section{Approximation by trigonometric polynomials}\label{app:trigonometric}
Results presented in this section are (rather straightforward) generalizations
of  Lemma~8.5.1 and Theorem~8.5.3 from~\cite{Saranen:2002:PIP}, valid in
two dimensions, to the multi-dimensional vector setting with different grid
spacings $h_\alp$.

\begin{proof}[Proof of \Lref{lem:approximation}]
Convergence in~\eqref{eq:conv_PN} is a consequence of the density of the set
of trigonometric polynomials $\{\bfun{\k}\}_{\k\in\Zd}$ in
$\Lper{2}{}$, e.g.~\cite[pp.~89--91]{rudin1986real}. 

As of~\eqref{eq:rate_of_conv_PN}, combining~\eqref{eq:sobolev_norm_Hs}
with~\eqref{eq:PN_def} reveals that
\begin{align*}
\norm{\Tu - \PN{\Tu}}{\Hper{r}{\xRd}}^2
& = 
\sum_{\k\in\Zd\setminus\ZNd} 
\norm{\Tz(\k)}{\xRd}^2
\norm{\FT{\Tu}(\k)}{\xCd}^2
\\
& = 
\sum_{\k\in\Zd\setminus\ZNd} 
\norm{\Tz(\k)}{\xRd}^{2(r-s)}
\norm{\Tz(\k)}{\xRd}^{2s}
\norm{\FT{\Tu}(\k)}{\xCd}^2
\\
& \leq 
\Ch^{2(s-r)}
\sum_{\k\in\Zd\setminus\ZNd} 
\norm{\Tz(\k)}{\xRd}^{2s}
\norm{\FT{\Tu}(\k)}{\xCd}^2
\leq 
\Ch^{2(s-r)}
\norm{\Tu}{\Hper{s}{\xRd}}^2.
\end{align*}

In order to prove~\eqref{eq:rate_of_conv_QN}, we first establish the Fourier
representation of operator $\QN{}$ for $\Tu \in \Hper{s}{\xRd}$ with $s >
\dime/2$ in the form
\begin{align}\label{eq:QN_Fourier}
\QN{\Tu}(\x)
=
\sum_{\k\in\ZNd}
\Bigl[
  \sum_{\m \in \Zd}
  \FT{\Tu}(\k+\m\odot\TN)
\Bigr]
\bfun{\k}(\x)
\text{ for }
\x \in \puc,
\end{align}
where $\odot$ denotes the element-by-element multiplication. Indeed, since
$\QN{}$ is a projection on $\TrigN{}$, we have\footnote{The operator $\QN{}$ is
applied to scalar functions in the same way as in~\eqref{eq:QN_def}.}
$\QN{\bfun{\k}} = \bfun{\k}$ for any $\k \in \ZNd$. Moreover,
\begin{align*}
\bfun{\k+\m \odot\TN}(\x_\TN^\Tl)
=
\varphi_{\k}(\x_\TN^\Tl)
\exp\left(2\imu\pi \sprod{\Tl}{\m}{\xRd} \right)
=
\bfun{\k}(\x_\TN^\Tl),
\end{align*}
for any $\k, \Tl \in \ZNd$ and $\m \in \Zd$, so that 
\begin{align*}
\QN{\bfun{\k+ \m\odot\TN}}
=
\QN{\varphi_{\k}}
=
\varphi_{\k}.
\end{align*}
Since $\QN{}$ is a linear operator, we arrive at
\begin{align*}
\QN{\Tu}
& =
\QN{%
\sum_{\k \in \Zd} \FT{\Tu}(\k)
\bfun{\k}
}
=
\QN{%
\sum_{\k \in \ZNd} 
\sum_{\m \in \Zd}
\FT{\Tu}(\k + \m \odot \TN)
\bfun{\k + \m \odot \TN}
}
\\
& =
\sum_{\k \in \ZNd} 
\sum_{\m \in \Zd}
\FT{\Tu}(\k + \m \odot \TN)
\bfun{\k}.
\end{align*}

Now we proceed to the last part of the proof. Orthogonality of
$\PN{}$ entails that~\cite[Theorem~4.11]{rudin1986real}
\begin{align*}
\norm{\Tu - \QN{\Tu}}{\Hper{r}{\xRd}}^2
=
\norm{\Tu - \PN{\Tu}}{\Hper{r}{\xRd}}^2
+
\norm{\QN{\Tu} - \PN{\Tu}}{\Hper{r}{\xRd}}^2,
\end{align*}
where the first term is controlled by~\eqref{eq:rate_of_conv_PN}, and the latter
one can be estimated by combining~\eqref{eq:sobolev_norm_Hs}, \eqref{eq:QN_Fourier}, \eqref{eq:PN_def}, and the Cauchy inequality as
\begin{align*}
& \norm{\QN{\Tu} - \PN{\Tu}}{\Hper{r}{\xRd}}^2
\\
&=
\sum_{\k\in\ZNd} \norm{\underline{\Tz}(\k)}{\xRd}^{2r}
\norm{
\sum_{\m \in \ZdmO}
\FT{\Tu}(\k + \m \odot \TN)}{\xCd}^{2}
\\
&\leq
\sum_{\k\in\ZNd}
\left( 
  \sum_{\m \in \ZdmO} 
  \frac{%
    \norm{\underline{\Tz}(\k)}{\xRd}^r}{%
    \norm{\Tz(\k + \m \odot \TN)}{\xRd}^s
  }
\norm{\Tz(\k + \m \odot \TN)}{\xRd}^s
\norm{\FT{\Tu}(\k + \m \odot \TN)}{\xCd}
\right)^2 
\\
&\leq
\sum_{\k \in \ZNd}
\left( 
  \sum_{\m \in \ZdmO} 
  \frac{%
    \norm{\underline{\Tz}(\k)}{\xRd}^{2r}}{%
    \norm{\Tz(\k + \m \odot \TN)}{\xRd}^{2s}
  }
\right)
\times
\\
& \quad
\left( 
  \sum_{\m \in \ZdmO}
  \norm{\Tz(\k + \m \odot \TN)}{\xRd}^{2s}
  \norm{\FT{\Tu}(\k + \m \odot \TN)}{\xCd}^2  
\right)
\\
&\leq
\varepsilon^2_\TN \norm{\Tu}{\Hper{s}{\xRd}}^2.
\end{align*}
The constant is provided by 
\begin{align*}
\varepsilon^2_\TN
& = 
\max_{\k\in\ZNd}
\left(
  \sum_{\m \in \ZdmO} 
  \frac{%
    \norm{\underline{\Tz}(\k)}{\xRd}^{2r}}{%
    \norm{\Tz(\k + \m \odot \TN)}{\xRd}^{2s}
  }
\right)
\\
&
\leq \dime^{r} \ch^{-2r}
\max_{\k\in\ZNd}
\left( 
  \sum_{\m \in \ZdmO} 
  \left[ 
    \sum_{\alp}
    (\frac{N_{\alp}}{2Y_{\alp}})^2
    |\frac{2k_{\alp}}{N_{\alp}} + 2m_{\alp}|^2 
  \right]^{-s} 
\right)
\\
& \leq 
\dime^{r} \ch^{-2r} \Ch^{2s}
\max_{\k\in\ZNd}
\left( 
  \sum_{\m \in \ZdmO} 
  \left[ 
    \sum_{\alp} |\frac{2k_{\alp}}{N_{\alp}} + 2m_{\alp}|^2 
  \right]^{-s} 
\right)
\\
& =
\dime^{r} \rh^{-2r} \Ch^{2(s-r)}
\max_{\k\in\ZNd}
\left( 
  \sum_{\m \in \ZdmO} 
  \left[ 
    \sum_{\alp} |\frac{2|k_{\alp}|}{N_{\alp}} + 2m_{\alp}|^2 
  \right]^{-s} 
\right)
\\
& \leq 
\dime^{r} \rh^{-2r} \Ch^{2(s-r)}
\left( 
  \sum_{\m \in \NOmO} 
  \norm{\m}{\xRd}^{-2s}
\right),
\end{align*}
where the last estimate generalizes the one-dimensional version
from~\cite[p.~243]{Saranen:2002:PIP} to

\begin{align*}
\max_{\k\in\ZNd}
  \sum_{\m \in \ZdmO} 
  \left[ 
    \sum_{\alp} |\frac{2|k_{\alp}|}{N_{\alp}} + 2m_{\alp}|^2 
  \right]^{-s} 
&
\leq \sum_{\m \in \ZdmO} 
\max_{\k\in\ZNd}
  \left[ 
    \sum_{\alp} |\frac{2|k_{\alp}|}{N_{\alp}} + 2m_{\alp}|^2 
  \right]^{-s} 
\\
&= \sum_{\m \in \ZdmO} \left[ 
    \sum_{\alp} |H(m_\alp) + 2m_{\alp}|^2 
  \right]^{-s}
   \\
  &= \sum_{\m \in \NOmO } 
  \norm{\m}{\xRd}^{-2s}
\end{align*}
with the Heaviside-like function defined as 
\begin{align*}
H(m_\alp) = 
\begin{cases}
0&\text{for }m_\alp\geq 0,
\\
1&\text{for }m_\alp < 0.
\end{cases}
\end{align*}
Noticing that the above sum is convergent for $s > \dime/2$, we obtain
\begin{align*}
\norm{\QN{\Tu} - \PN{\Tu}}{\Hper{r}{\xRd}}^2
\leq
\dime^{r} \rh^{-2r} \Ch^{2(s-r)}
\left( 
  \sum_{\m \in\NOmO} 
  \norm{\m}{\xRd}^{-2s}
\right)
\norm{\Tu}{\Hper{s}{\xRd}},
\end{align*}
and the proof of \Lref{lem:approximation} now follows directly
from~\eqref{eq:rate_of_conv_PN}.
\end{proof}

\section{Regularity result}\label{app:regularity}
In order to justify the requirements on the weak solution to the periodic
cell problem, needed to establish the rate of convergence of Galerkin
approximations in \Sref{sec:Discretization}, in this section we collect
basic regularity results. To this purpose, we employ the well-known techniques
based on difference quotients, e.g. Theorem~3 of Section~5.8.2 and Theorem 1 of
Section~6.3.1 in \cite{Evans:2000:PDE}, simplified due to the periodic setting.
Moreover, to keep the exposition compact, we treat only the case $\TA \in
\Wper{1}{\infty}{\xRdd}$. The general statement, i.e. $\TA \in
\Wper{s}{\infty}{\xRdd}$ with $s \in \set{N}$ implies that $\Tefl \in
\Hper{s}{\xRd}$, follows by induction; a proof based on Theorem~2 of
Section~6.3.1 in \cite{Evans:2000:PDE} is available in
\cite[pp.~113--114]{Vondrejc:2013:FFT}.
 
In particular, the $\alp$-th difference coefficient of a function $\Tf \in
\Lper{2}{\X}$ with the step $\step \in \R$ is provided by
\begin{align*}
\DQf{\alp} \Tf(\x)
=
\frac{\Tf(\x + \step \Tunv{\alp}) - \Tf(\x)}{\step}
\text{ for }
\x \in \puc,
\end{align*}
where $\Tunv{\alp}=(\del_{\alp\beta})$. The following result summarizes the
relation between difference quotients and periodic Sobolev functions.

\begin{lemma}\label{lem:diff_sobolev}
  Assume that $\Tu \in \Lper{2}{\xRd}$ and that there exist $C$ independent of
  $\step$ and $\alp$ such that $\norm{\DQf{\alp}\Tu}{\Lper{2}{\xRd}} \leq C$.
  Then $\Tu \in \Hper{1}{\xRd}$.
\end{lemma}

\begin{proof} The demonstration rests on the integration by parts
formula for difference quotients
\begin{align*}
\sprod{\Tu}{\DQf{\alp}\Tu}{\Lper{2}{\xRd}}
=
-
\sprod{\DQb{\alp}\Tu}{\Tu}{\Lper{2}{\xRd}}
\end{align*}
valid for all $\Tu \in \Hper{1}{\xRd}$. Since $\DQf{\alp}\Tu$ is uniformly
bounded in $\Lper{2}{\xRd}$, we can extract a subsequence~(not relabeled) such
that $\DQb{\alp} \Tu \rightharpoonup \Tv$ weakly in $\Lper{2}{\xRd}$ as $\step
\rightarrow 0$. Therefore,
\begin{align*}
\sprod{\Tu}{\D{\alp}\Tu}{\Lper{2}{\xRd}}
&=
\lim_{h \rightarrow 0}
\sprod{\Tu}{\DQf{\alp}\Tu}{\Lper{2}{\xRd}}
\\
&=
- 
\lim_{h \rightarrow 0}
\sprod{\DQb{\alp}\Tu}{\Tu}{\Lper{2}{\xRd}}
=
\sprod{\Tv}{\Tu}{\Lper{2}{\xRd}},
\end{align*}
so that $\Tv = \D{\alp}\Tu$, $\nabla \Tu \in \Lper{2}{\xRdd}$ and consequently
$\Tu \in \Hper{1}{\xRd}$.
\end{proof}

\begin{lemma}\label{lem:regularity_result}
Let $\TA \in \Wper{1}{\infty}{\xRdd}$ satisfy~\eqref{eq:A2} and $\Tefl \in
\Lper{2}{\xRd}$ be the weak solution to the cell
problem~\eqref{eq:weak_solution}. Then, $\Tefl \in \Hper{1}{\xRd}$ and we
have
\begin{align}\label{eq:a-priory}
\norm{\Tefl}{\Hper{1}{\xRd}}
\leq
\frac{1 + \rA}{\cA}
\norm{\TA}{\Wper{1}{\infty}{\xRdd}}
\norm{\TE}{\xRd}.
\end{align}
\end{lemma}

\begin{proof}
Test the formulation~\eqref{eq:weak_solution} with a function
$\Tv = -\DQb{\alp}(\DQf{\alp} \Tefl)$ to obtain
\begin{align}\label{eq:proof_B1}
-
\sprod{\TA\Tefl}{\DQb{\alp}(\DQf{\alp} \Tefl)}{\Lper{2}{\xRd}}
=
\sprod{\TA\TE}{\DQb{\alp}(\DQf{\alp} \Tefl)}{\Lper{2}{\xRd}}.
\end{align}
Utilizing simple relations
\begin{align*}
\sprod{\Tv}{\DQb{\alp}\Tu}{\Lper{2}{\xRd}}
& =
-
\sprod{\DQf{\alp}\Tv}{\Tu}{\Lper{2}{\xRd}},
&
\DQf{\alp}(\TA\Tu)
=
\TA^\step ( \DQf{\alp}\Tu )
+
\DQf{\alp}(\TA) \Tu,
\end{align*}
with $\TA^{\step,\alp}(\x) = \TA(\x+\step \Tunv{\alp})$
\eqref{eq:proof_B1} transfers to 
\begin{align*}
\sprod{\TA^\step \DQf{\alp}\Tefl}{\DQf{\alp}\Tefl}{\Lper{2}{\xRd}}
=
-
\sprod{\DQf{\alp}(\TA)\TE}{\DQf{\alp}{\Tefl}}{\Lper{2}{\xRd}} 
-
\sprod{\DQf{\alp}(\TA)\Tefl}{\DQf{\alp}\Tefl}{\Lper{2}{\xRd}}.
\end{align*}
Property~\eqref{eq:A1} and the H\"{o}lder inequality imply
\begin{align*}
\cA \norm{\DQf{\alp}\Tefl}{\Lper{2}{\xRd}}^2
&\leq 
\norm{\DQf{\alp}\TA}{\Lper{2}{\xRdd}} \left(
\norm{\TE}{\xRd}
+
\norm{\Tefl}{\Lper{2}{\xRd}}
\right)
\norm{\DQf{\alp}\Tefl}{\Lper{2}{\xRd}};
\end{align*}
since material coefficients $\TA \in \Wper{1}{\infty}{\xRdd}$ are Lipschitz
continuous, it holds
\begin{align*}
\cA \norm{\DQf{\alp}\Tefl}{\Lper{2}{\xRd}}
\leq 
\norm{\TA}{\Wper{1}{\infty}{\xRdd}}
\left(
\norm{\TE}{\xRd}
+
\norm{\Tefl}{\Lper{2}{\xRd}}
\right).
\end{align*}
The difference quotient $\DQf{\alp}\Tefl$ is thus bounded independently of
$\step$, as required by \Lref{lem:diff_sobolev}, so that $\Tefl \in
\Hper{1}{\xRd}$. The inequality~\eqref{eq:a-priory} now follows from standard
a-priory estimates on $\Tefl$, i.e. $\norm{\Tefl}{\Lper{2}{\xRd}} \leq \rA
\norm{\TE}{\xRd}$.
\end{proof}


\bibliography{liter}

\end{document}